\documentclass[12pt]{article}

\usepackage[T1]{fontenc}
\usepackage{lmodern}
\usepackage{textcomp}

\usepackage{amsmath}%
\usepackage{amssymb}%
\usepackage[table]{xcolor}%

\usepackage{todonotes}
\usepackage[in]{fullpage}%

\usepackage[amsmath,thmmarks]{ntheorem}%

\theoremseparator{.}%

\usepackage{titlesec}%
\titlelabel{\thetitle. }%
\usepackage{xcolor}%
\usepackage{mleftright}%
\usepackage{xspace}%
\usepackage{graphicx}
\usepackage{hyperref}%
\usepackage{mathtools}

\usepackage{hyperref}%
\hypersetup{%
      unicode,
      breaklinks,%
      colorlinks=true,%
      urlcolor=[rgb]{0.25,0.0,0.0},%
      linkcolor=[rgb]{0.5,0.0,0.0},%
      citecolor=[rgb]{0,0.2,0.445},%
      filecolor=[rgb]{0,0,0.4},
      anchorcolor=[rgb]={0.0,0.1,0.2}%
}
\usepackage[ocgcolorlinks]{ocgx2}

\usepackage{lmodern}

\theoremseparator{.}%

\theoremstyle{plain}%
\newtheorem{theorem}{Theorem}[section]

\newtheorem{lemma}[theorem]{Lemma}
\newtheorem{conjecture}[theorem]{Conjecture}

\newtheorem{fact}[theorem]{Fact}

\newtheorem{proposition}[theorem]{Proposition}
\newtheorem{prop}[theorem]{Proposition}

\newtheorem{remark}[theorem]{Remark}

\newtheorem*{remark:unnumbered}[theorem]{Remark}%
\newtheorem*{remarks}[theorem]{Remarks}%

\newtheorem{definition}[theorem]{Definition}

\newtheorem{assumption}[theorem]{Assumption}%

\newcommand{\myqedsymbol}{\rule{2mm}{2mm}}

\theoremheaderfont{\em}%
\theorembodyfont{\upshape}%
\theoremstyle{nonumberplain}%
\theoremseparator{}%
\theoremsymbol{\myqedsymbol}%
\newtheorem{proof}{Proof:}

\newenvironment{proofof}[1]
{\par\noindent\textit{Proof of #1:}}
{\hfill$\myqedsymbol$\par}

\providecommand{\emphind}[1]{}%
\renewcommand{\emphind}[1]{\emph{#1}\index{#1}}

\definecolor{blue25emph}{rgb}{0, 0, 11}

\providecommand{\emphic}[2]{}
\renewcommand{\emphic}[2]{\textcolor{blue25emph}{%
      \textbf{\emph{#1}}}\index{#2}}

\providecommand{\emphi}[1]{}%
\renewcommand{\emphi}[1]{\emphic{#1}{#1}}

\definecolor{almostblack}{rgb}{0, 0, 0.3}

\providecommand{\emphw}[1]{}%
\renewcommand{\emphw}[1]{{\textcolor{almostblack}{\emph{#1}}}}%

\providecommand{\emphOnly}[1]{}%
\renewcommand{\emphOnly}[1]{\emph{\textcolor{blue25}{\textbf{#1}}}}

\newcommand{\HLink}[2]{\hyperref[#2]{#1~\ref*{#2}}}
\newcommand{\HLinkSuffix}[3]{\hyperref[#2]{#1\ref*{#2}{#3}}}

\providecommand{\deflab}[1]{}
\renewcommand{\deflab}[1]{\label{def:#1}}

\providecommand{\eqlab}[1]{}%
\renewcommand{\eqlab}[1]{\label{equation:#1}}

\newcommand{\remove}[1]{}%

\newcommand{\R}{\mathbb{R}}
\newcommand{\dd}{\mathrm{d}}

\usepackage[inline]{enumitem}

\newlist{compactenumA}{enumerate}{5}%
\setlist[compactenumA]{topsep=0pt,itemsep=-1ex,partopsep=1ex,parsep=1ex,%
   label=(\Alph*)}%

\newlist{compactenuma}{enumerate}{5}%
\setlist[compactenuma]{topsep=0pt,itemsep=-1ex,partopsep=1ex,parsep=1ex,%
   label=(\alph*)}%

\newlist{compactenumI}{enumerate}{5}%
\setlist[compactenumI]{topsep=0pt,itemsep=-1ex,partopsep=1ex,parsep=1ex,%
   label=(\Roman*)}%

\newlist{compactenumi}{enumerate}{5}%
\setlist[compactenumi]{topsep=0pt,itemsep=-1ex,partopsep=1ex,parsep=1ex,%
   label=(\roman*)}%

\newlist{compactitem}{itemize}{5}%
\setlist[compactitem]{topsep=0pt,itemsep=-1ex,partopsep=1ex,parsep=1ex,%
   label=\ensuremath{\bullet}}%

\numberwithin{equation}{section}%

\newcommand{\di}{\displaystyle\int}

\newcommand{\oc}{\Omega^c}
\newcommand{\om}{\Omega}
\newcommand{\dic}{\displaystyle\int_{\oc}}
\newcommand{\po}{\partial\Omega}

\begin{document}

\title{Low-Temperature Asymptotics of the Poincaré and the log-Sobolev Constants for Łojasiewicz Potentials}

\author{%
   Aziz Ben Nejma\thanks{%
   \parbox[t]{0.95\linewidth}{%
Université Paris Cité and Sorbonne Université, CNRS, Laboratoire de Probabilités, Statistique et Modélisation, F-75013 Paris, France and DMA, École normale supérieure, Université PSL, CNRS, 75005 Paris, France.%
}}
}

\maketitle

\setcounter{secnumdepth}{3}
\begin{abstract}
    In this paper, we establish the low-temperature asymptotics of the Poincaré inequality constant for a class of convex potentials satisfying a Łojasiewicz inequality. In addition, we disprove a conjecture previously posed by Chewi and Stromme on the low-temperature asymptotics of the log-Sobolev constant and determine the correct asymptotic behavior in dimension one.
\end{abstract}

\section{Introduction}
Let $\mu = \frac{1}{Z}e^{-V}$ be a probability measure defined on $\mathbb{R}^n$. The measure $\mu$ satisfies a Logarithmic Sobolev (or log-Sobolev) Inequality (LSI)  with constant $C$ if the inequality
\begin{equation}\label{lsidef}
\operatorname{Ent}_\mu(h^2) \coloneqq \di h^2\log(h^2)\,\dd \mu - \bigg(\di h^2\,\dd \mu \bigg)\log\bigg(\di h^2\,\dd \mu\bigg) \leq 2C\di \lVert \nabla h \rVert^2\,\dd \mu
\end{equation}
holds whenever $h$ is a compactly supported smooth function.

\noindent The measure $\mu$ satisfies a Poincaré inequality with constant $C$ if the similar inequality
$$\text{Var}_\mu(h) \coloneqq\di h^2\,\dd \mu - \Big(\di h\,\dd \mu \Big)^2 \leq C\di \lVert \nabla h \rVert ^2\,\dd \mu.$$
holds for all compactly supported smooth functions $h$.
We will denote by $C_\mathrm{LS}(\mu)$ and $C_{\mathrm{P}}(\mu)$ the best constants in the two previous inequalities, satisfying $C_{\mathrm{P}}(\mu) \leq C_\mathrm{LS}(\mu)$.
Among other applications, these two constants determine the exponential rates of convergence to equilibrium of the overdamped Langevin diffusion in relative entropy and in $L^2$ distance. Define $L = \Delta - \nabla V \cdot \nabla$ and let $P_t = e^{tL}$ be the semigroup corresponding to the Markov process $$\,\dd X_s = -\nabla V(X_s)\,\dd s+\sqrt{2}\,\mathrm{d}B_s$$
then it holds that \cite{bakry2014analysis}
$$\lVert P_tf - \bar f \rVert_{L^2(\mu)}^2 \leq e^{-\frac{t}{C_{\mathrm{P}}(\mu)}}\lVert f - \bar f \rVert_{L^2(\mu)}$$ and that
$$\operatorname{Ent}_\mu(P_tf) \leq e^{-\frac{2t}{C_{\mathrm{LS}}(\mu)}}\operatorname{Ent}_\mu (f)$$
for all functions $f \in L^2(\mu)$, where $\bar f = \di f\,\dd \mu$.

Interest in the asymptotic behavior of these constants in the low-temperature regime dates back to the Eyring-Kramers formula and to the Arrhenius law \cite{eyring1935activated,kramers1940brownian,bovier2004metastability}~: this regime is defined by considering measures $\mu_t$ of the form $\frac{1}{Z_t}e^{\frac{-V}{t}}$ where $V$ is a given potential and where we let $t \to 0$. The low-temperature asymptotics of the Poincaré and the Logarithmic Sobolev inequalities usually depend on the optimization landscape of the potential $V$, as $\mu_t$ concentrates near the minimizers of the potential $V$ when $t \to 0$.

In the setting of simulated annealing, the asymptotics of $C_{\mathrm{P}}(\mu_t)$ and $C_{\mathrm{LS}}(\mu_t)$ have also been extensively investigated, as these constants control the mixing times and capture the metastable behavior of the overdamped Langevin dynamics, hence dictating the correct cooling rates required for convergence (see, e.g., \cite{HOLLEY1989333, Miclo1992}).

While Eyring-Kramers-type formulas describe the blow-up of these constants, for instance when the potential $V$ has several local minima, or when the set of minimizers is disconnected, the case where these constants tend to $0$ has also been investigated. The simplest case in which this occurs is when the potential is strongly convex. Indeed, the Bakry-Émery criterion (Proposition \ref{be}) implies that $C_{\mathrm{LS}}(\mu_t)=O(t)$ and that $C_{\mathrm{P}}(\mu_t)=O(t)$ whenever $V$ is strongly convex. The exact asymptotics of $C_{\mathrm{P}}(\mu_t)$ and $C_{\mathrm{LS}}(\mu_t)$ have been established in the recent work \cite{chewi2024ballisticlimitlogsobolevconstant} for a class of functions including strongly convex potentials, more precisely,
\begin{theorem}[Chewi-Stromme]\label{CS}
    Let $V \in C^2(\mathbb{R}^n)$ be a potential having a unique minimizer $x_0$ and satisfying a Polyak-Łojasiewicz inequality, namely $$V-\inf V \leq C_{PL} \lVert\nabla V \rVert^2$$
    where $C_{PL}$ is the best constant in the last inequality. Assume $\dfrac{\Delta V}{1+\lVert \nabla V\rVert^2}$ is bounded above, then
    $$\dfrac{C_{\mathrm{LS}}(\mu_t)}{t}\xrightarrow[]{t \to 0^+} 2C_{PL} $$
    and
    $$\dfrac{C_{\mathrm{P}}(\mu_t)}{t} \xrightarrow[]{t \to 0^+} \dfrac{1}{\lVert\nabla ^2V(x_0) \rVert_{op}}.$$
\end{theorem}

Under the assumptions of Theorem \ref{CS}, the measures $\mu_t$ converge weakly to $\mu_0 \coloneqq \delta_{x_0}$ as $t \to 0$. Therefore, an equivalent formulation of Theorem \ref{CS} is that 
$$\dfrac{C_{\mathrm{LS}}(\mu_t)-C_{\mathrm{LS}}(\mu_0)}{t}\xrightarrow[]{t \to 0^+} 2C_{PL}.$$
This formulation can be easily generalized for potentials not having a unique minimizer, i.e. when the measures $\mu_t$ do not converge to a Dirac mass but to another measure. That is how the authors of \cite{chewi2024ballisticlimitlogsobolevconstant} formulated the following conjecture.
\begin{conjecture}[Chewi-Stromme]\label{conj}
    Let $V$ be a potential satisfying a Polyak-Łojasiewicz inequality such that $\mu_t \to \mu_0$ where $\mu_0$ is a probability measure. Then
    $$\dfrac{C_{\mathrm{LS}}(\mu_t)-C_{\mathrm{LS}}(\mu_0)}{t}\xrightarrow[]{t \to 0^+} 2C_{PL}.$$
\end{conjecture}

\paragraph*{Main results.}
In this paper we disprove Conjecture \ref{conj} (in Subsection \ref{refutation}). Besides, we establish the exact asymptotics of the low-temperature Poincaré constants for convex potentials having a whole domain of minimizers, and that satisfy a Łojasiewicz-type inequality.
A key example is given by potentials of the form $x\mapsto \mathrm{dist}(x,\Omega)^2$ where $\Omega$ is a convex bounded domain.
Moreover, we determine the correct asymptotic behavior of the log-Sobolev constant for general Polyak-Łojasiewicz potentials in dimension $1$.

Before stating the first main theorem, we set up the framework (see the discussion of the assumptions in Section \ref{prelim}.

\begin{assumption}\label{assumptions}
We make the following assumptions :

\begin{itemize}
    \item $V \in C^{1}_{loc}(\mathbb{R}^n)$ is a convex nonnegative potential
    \item $\min V=0$ without loss of generality
    \item $\mathrm{argmin} V =\bar\Omega$ where $\om$ is a bounded smooth domain
    \item $V(x) \leq C\mathrm{dist}(x,\om)^\beta$ on a neighborhood of $\om$, for some $C,\beta >0$.
    \item There exists a positive constant $\alpha < 2\beta$ and a function $a \in L^\infty(\partial\om)\backslash\{0\}$ such that $$V(x)=\dfrac{\mathrm{dist}(x,\om)^\alpha}{a(p(x))+o(1)}\quad \text{as }\mathrm{dist}(x,\om) \to 0^+$$ where $p$ is the orthogonal projection onto the convex set $\om$.
\end{itemize}
\end{assumption}
\begin{remark}
\normalfont
Because of convexity, we necessarily have $\alpha\geq 1$ and $a \geq 0$.
Note that Assumption \ref{assumptions} allows $a(p(x))$ to vanish, as long as $a$ is not identically zero. Moreover, if $\alpha<2$, since $V$ is $C^1$, the inequality $V(x) \leq C \mathrm{dist}(x,\om)^\beta$ is automatically satisfied near $\om$ for $\beta=1$.
\end{remark}
\begin{remark}
\normalfont
    The final part of Assumption \ref{assumptions} is a variant of a Łojasiewicz inequality : we do not only require $V \gtrsim \mathrm{dist}(\cdot,\om)^\alpha$, but our arguments need $V$ to grow exactly like a multiple of $\mathrm{dist}(\cdot,\om)^\alpha$ in the directions pointing outwards from $\om$ for which this last quantity is the correct lower bound.
\end{remark}
\begin{remark}
    \normalfont
    Assumption \ref{assumptions} is satisfied in particular when $V$ is convex such that $V(x)\sim \mathrm{dist}(x,\om)^\alpha$ near $\om$, which is the case for all functions of the form $\mathrm{dist}(\cdot,\om)^\alpha$ for $\alpha > 1$.
\end{remark}

Our first main theorem is the following :
\begin{theorem}\label{thm1}
    Let $V$ be  a potential satisfying Assumption \ref{assumptions}, then it holds that
    $$\dfrac{C_{\mathrm{P}}(\mu_t)-C_{\mathrm{P}}(\mu_0)}{t^\frac{1}{\alpha}} \xrightarrow[t\to 0]{} \Lambda_{\om,V} \in \mathbb{R}$$
    where $\Lambda_{\om,V}$ is an explicit constant (see Theorem \ref{mainthm}).
\end{theorem}

\begin{remark}
\normalfont
    $\Lambda_{\Omega,V}$ is generically nonzero, as shown by the one-dimensional example in Subsection \ref{refutation} and Remark \ref{lam}.
\end{remark}
\begin{remark}
    \normalfont
    Although it should not be essential, the convexity of $V$ is very convenient when proving Theorem \ref{mainthm} because it allows to control $\lVert \nabla ^2f\rVert$ with $\Gamma_2(f,f)$ (see Definition \ref{carreduchamp}) whenever $f$ is a smooth function (see Proposition \ref{bochner} and Lemma \ref{pasencore} ). In addition, it enables us to bound the Poincaré constant of $e^{-V}$ by the variance of the normalized measure (see Proposition \ref{kls1d} and Lemma \ref{klslemma}).
\end{remark}

Our second main result is about the low-temperature asymptotics of the log-Sobolev constant in dimension $1$. In this regime, $C_{\mathrm{LS}}(\mu_t)-C_{\mathrm{LS}}(\mu_0)$ scales like $\sqrt{t}$, with the precise behavior depending on the behavior of $V$ near the boundary of the minimizing set.

\begin{theorem}\label{lsi1d}
    Let $V \colon \mathbb{R} \to \mathbb{R}$ a potential satisfying a Polyak-Łojasiewicz inequality, having multiple minimizers. Let $a<b$ be real numbers such that $[a,b] = \mathrm{argmin}(V)$. Assume $V_{|[b,+\infty)}$ and $V_{|(-\infty,a]}$ are $C^2$ and are such that $\dfrac{V''}{1+V'^2}$ is bounded above on $\mathbb{R}\backslash[a,b]$, then
    \begin{equation*}
        \dfrac{C_{\mathrm{LS}}(\mu_t)-C_{\mathrm{LS}}(\mu_0)}{\sqrt{t}}\xrightarrow{t\to 0^+}\sqrt{2}(b-a)\pi^{-\frac{3}{2}}\big(V''(a^-)^{-\frac{1}{2}}+{V''(b^+)^{-\frac{1}{2}}}\big).   
    \end{equation*}
\end{theorem}

\begin{remark}
 \normalfont
    Although we expect the result of Theorem \ref{lsi1d} to remain valid in any dimension (up to minor adjustments), our proof relies heavily on several properties that are specific to dimension $1$. Namely, we use the fact that there are no extremal functions for the logarithmic Sobolev inequality on a line segment (see Remark \ref{extremal}), a feature that fails for generic convex domains in higher dimensions \cite{Rothaus1981Diffusion}. The proof also makes essential use of the fact that $\mathbb{R}\backslash[a,b]$ is a finite union of convex sets (in fact, exactly $2$) to apply the Rothaus lemma and to bound various integrals over $\mathbb{R}\backslash [a,b]$.
\end{remark}

\paragraph{Acknowledgements.} I am deeply grateful to Max Fathi for his guidance, support and very helpful suggestions.\\
The author has received support under the program "Investissement d'Avenir" launched by the French Government and implemented by ANR, with the reference ANR-18-IdEx-0001 as part of its program "Emergence".  He was also supported by the Agence Nationale de la Recherche (ANR) Grant ANR-23-CE40-0003 (Project CONVIVIALITY).

\section{Preliminaries on functional inequalities}\label{prelim}
In this section, we recall the basic results that we will use throughout the remainder of this work.

\paragraph{The Poincaré inequality as a spectral gap.}
The Poincaré inequality can be seen as a consequence of a spectral gap for a diffusion operator \cite{bakry2014analysis} : let $\mu =\frac{1}{Z} e^{-V}$ be a probability measure satisfying a Poincaré inequality and $L_V = \Delta -\nabla V\cdot \nabla$ a diffusion operator, then the operator $(-L_V)$ is symmetric in $L^2(\mu)$ and has (under mild assumptions) spectrum $0= \lambda_0 < \lambda_1 \leq \dots$,
where $\lambda_0=0$ corresponds to constant functions. Therefore,
$$\lambda_1=\inf\Bigg\{\dfrac{\di \lVert\nabla f\rVert^2\,\dd \mu}{\di f^2\,\dd \mu} \colon f \in H^1(\mu), \di f\,\dd \mu=0\Bigg\}.$$
Hence $C_{\mathrm{P}}(\mu)=\dfrac{1}{\lambda_1}$, and the Poincaré constant reduces to the inverse of the spectral gap.
When $\mu$ is the uniform measure on a smooth bounded domain, then $C_{\mathrm{P}}(\mu)$ also corresponds to the inverse of the spectral gap of the Neumann Laplacian.

\vspace{10pt}
When the potential $V$ is convex, the Poincaré constant is always finite and controlled by the variance of the measure $\mu$ \cite{bobkov}, and we will need the following proposition later.
\begin{proposition}\label{kls1d}
Let $\mu$ be a log-concave probability measure on $\mathbb{R}^n$, then
    $$C_P(\mu) \leq C_n\di \left\Vert x - \int y\,\dd \mu(y)\right\Vert^2 \,\dd \mu(x)$$
where $C_n$ is a constant depending only on the dimension $n$.
\end{proposition}
Recent developments by Klartag \cite{Klartag2023LogarithmicBF, Klartag2024IsoperimetricII} show that $C_n = O(\log(n))$.

\paragraph{The Bakry-Émery criterion and the $\Gamma$-calculus.} In the context of functional inequalities, the Bakry-Émery criterion is one of the simplest ways to show log-Sobolev or Poincaré inequalities. This criterion goes back to Bakry and Émery \cite{Bakry1985}, with many developments since then \cite{bakry2014analysis}.
\begin{definition}\label{carreduchamp}
    Given a diffusion operator $L$, the ``carré du champ'' operator $\Gamma$ is defined as $$2\Gamma(f,g) = L(fg)-fLg-gLf$$
    for functions $f$ and $g$. The $\Gamma_2$ operator is then defined as 
    $$2\Gamma_2(f,g)=L\Gamma(f,g)-\Gamma(f,Lg)-\Gamma(g,Lf).$$
    A diffusion operator $L$ is said to satisfy a curvature-dimension condition $CD(\rho, \infty)$ if $$\Gamma_2(f,f) \geq \rho \Gamma(f,f)$$ for all smooth compactly supported functions $f$.
\end{definition}
A curvature-dimension condition $CD(\rho, \infty)$ satisfied by a diffusion operator implies a Poincaré and a log-Sobolev inequalities for the corresponding measure \cite{bakry2014analysis}. Here is the precise statement.
\begin{proposition}[Bakry-Émery]\label{be}
    Let $L_V=\Delta - \nabla V \cdot \nabla$ be a diffusion operator satisfying a curvature-dimension condition $CD(\rho, \infty)$ and let $\mu = \frac{1}{Z}e^{-V}$ be a probability measure, then
    $$C_{\mathrm{P}}(\mu) \leq \frac{1}{\rho} \text{ and }
     C_{\mathrm{LS}}(\mu) \leq \frac{1}{\rho}.$$
\end{proposition}

The $CD(\rho, \infty)$ condition is automatically satisfied when $V$ is $\rho$-strongly convex. In fact, for regular functions $f$, it holds that \cite{bakry2014analysis}
\begin{equation} 
  \begin{cases}
  \Gamma(f,f) = \lVert \nabla f \rVert ^2\\
  \Gamma_2(f,f)=  \langle\nabla ^2V \nabla f ,\nabla f\rangle + \lVert \nabla^2f\rVert^2_F
\end{cases}
\nonumber
\end{equation}
where $\lVert \cdot\rVert_F$ is the Frobenius (or Hilbert-Schmidt) norm on matrices. An additional consequence of the two equalities above, that will be useful in the sequel, is obtained when integrating $\Gamma_2(f)$ with respect to the measure $\mu = e^{-V}$ and using the fact that the operator $L_V$ is symmetric in $L^2(e^{-V})$ \cite{bakry2014analysis}.
\begin{proposition}[Bochner's integrated formula]\label{bochner}
Keeping the notation from Proposition \ref{be}, let $f$ be a sufficiently regular function, then 
$$\di (L_Vf)^2\,\dd \mu=\di \Gamma_2(f,f)\,\dd \mu = \di \langle\nabla ^2V \nabla f ,\nabla f\rangle + \lVert \nabla^2f\rVert^2_F\,\dd \mu.$$
\end{proposition}

\paragraph{Comparing log-Sobolev and Poincaré constants}
A classical linearization argument \cite{bakry2014analysis} shows that $C_{\mathrm{P}}(\mu) \leq C_{\mathrm{LS}}(\mu)$, and therefore LSI is stronger than the Poincaré inequality. The inequality is strict for a generic measure. However, the uniform measure on a line segment provides an example where equality holds \cite{gentil2004logsob}. Indeed,
\begin{proposition}\label{segment}
    Let $a>0$ and $\mu_a$ be the uniform measure on $[0,a]$, then
    $$C_{\mathrm{LS}}(\mu_a)=C_{\mathrm{P}}(\mu_a)=\dfrac{a^2}{\pi^2}.$$
\end{proposition}

\begin{remark}\label{extremal}
\normalfont
    As for line segments, uniform measures on convex sets with finite volume satisfy a Poincaré inequality and a log-Sobolev inequality. That is why $C_{\mathrm{LS}}(\mu_0)$ and $C_{\mathrm{P}}(\mu_0)$ (in Conjecture \ref{conj} and Theorem \ref{thm1}) are well-defined as long as Assumption \ref{assumptions} is satisfied.
\end{remark}

\begin{remark}\label{weiss}
\normalfont
There are no extremal functions for the log-Sobolev inequality associated with the uniform probability measure on a line segment. This fact can already be deduced from the proof of the inequality itself (see, e.g., \cite{WEISSLER1980218, gentil2004logsob}).
For a general probability measure $\mu$ such that $C_{\mathrm{LS}}(\mu)$ is finite, the inequality above, i.e. $C_{\mathrm{P}}(\mu)\leq C_{\mathrm{LS}}(\mu)$, is strict. In this case, a theorem of Rothaus \cite{Rothaus1981Diffusion} ensures the existence of an extremal function for the log-Sobolev inequality associated with $\mu$, at least when $\mu$ is compactly supported.
\end{remark}

\paragraph{About the assumptions and the Polyak-Łojasiewicz inequality.}
The Polyak-Łojasiewicz (PŁ) inequality governs the exponential rate of convergence of a function's gradient flow. Indeed,
\begin{fact}
The two following statements are equivalent for a function $V$ :
\begin{enumerate}
    \item $V$ satisfies a Polyak-Łojasiewicz inequality with constant $C$
    \item $V(y_t)-\min(V) \leq e^{-\frac{t}{C}}(V(y_0)-\min V)$ where $t\mapsto y_t$ is the gradient flow of $V$ starting from $y_0$.
\end{enumerate}
\end{fact}

All strongly convex functions satisfy a Polyak-Łojasiewicz inequality \cite{bolte2010characterizations}, but in general, PŁ functions do not have a unique minimizer. For instance, all functions of the form $x \mapsto \mathrm{dist}(x,F)^2$ where $F$ is a closed subset of $\mathbb{R}^n$ satisfy a Polyak-Łojasiewicz inequality (see, e.g., \cite{garrigos2023square}). Yet, it turns out that $C^2$ PŁ functions cannot have an arbitrary set of minimizers \cite{criscitiello2025} : for example, if this set is bounded, it must necessarily be a single point. Therefore, the assumption of uniqueness of the minimizer in Theorem \ref{CS} can be removed as it is already implied by the regularity assumption and the PŁ condition. This also means that, in order to handle PŁ functions not having a unique minimizer, the regularity assumption has to be weakened, as $C^2$ regularity entails much more rigidity than $C^{1,1}_{loc}$ for PŁ functions.
The $C^{1}_{loc}$ assumption is satisfied by all functions mentioned above, of the form $x \mapsto \mathrm{dist}(x,\om)^2$ and more generally of the form $x \mapsto \mathrm{dist}(x,\om)^\alpha$ for $\alpha>1$ when $\om$ is a convex set\footnote{Interestingly, $\om$ being convex is the only case where square distance functions are $C^{1,1}_{loc}$  \cite{PoliquinRockafellarThibault2000,nejma2025polyaklojasiewiczinequalityessentiallygeneral}.}.

Moreover, a well-known fact about PŁ functions is the following quadratic growth condition :
\begin{proposition}
    Let $V$ be a function satisfying a Polyak-Łojasiewicz inequality with constant $C_{PL}$, then for all $x \in \mathbb{R}^n$,
    $$\frac{1}{4C_{PL}}\mathrm{dist}(x,\mathrm{argmin}V)^2\leq V(x)-\min V.$$
\end{proposition}
This quadratic growth condition makes square distance functions a fundamental example among PŁ functions as in the smooth case, a PŁ function is bounded below and above, up to a a multiplicative constant, by the square distance function to the set of its minimizers. Remarkably, the log-Sobolev inequality can itself be interpreted as a Polyak-Łojasiewicz inequality in the Wasserstein metric for the relative entropy functional, see \cite{BLANCHET20181650} for a more detailed exposition.

\vspace{5pt}
Łojasiewicz inequalities \cite{lojasiewicz1963propriete} are a generalization of the PŁ condition. A function $f$ is said to satisfy a Łojasiewicz inequality if $f-\min f \gtrsim \mathrm{dist}(\cdot, \mathrm{argmin}f)^\alpha$ where $\alpha=2$ corresponds to the Polyak-Łojasiewicz case. Many functions, including all analytic and subanalytic functions, satisfy, at least near their minimizers, Łojasiewicz-type inequalities.

\vspace{5pt}
Additionally, in order to state Theorem \ref{thm1}, since  $\mu_t$ is a probability measure for $t>0$, we need that $e^{-\frac{V}{t}} \in L^1$ and $\om$ must have a finite Lebesgue measure. The potential $V$ being convex, $\om$ is also convex. Accordingly, $\om$ has to be bounded.

\section{Low-temperature asymptotics of Poincaré constants}
In this section, we disprove Conjecture \ref{conj} and we prove a precise version of Theorem \ref{thm1}.

\subsection{A refutation of Conjecture \ref{conj}}\label{refutation}
The counterexample that we provideis elementary and one-dimensional, yet it gives the asymptotic behavior that we believe to be correct (see Section \ref{LSIasympt}).
\paragraph{General heuristics.} 
Having in mind Theorem \ref{thm1} for the Poincaré constant and the following heuristics, in addition to other arguments that we leave to future work, it is natural to expect the asymptotics of $\small{\dfrac{C_{\mathrm{LS}}(\mu_t)-C_{\mathrm{LS}}(\mu_0)}{\sqrt{t}}}$ to be of order $\sqrt{t}$, when the potential $V$ is Polyak-Łojasiewicz. In fact, given a regular function $g \in C^{\infty}_c(\mathbb{R}^d)$, if $V$ has a bounded domain of minimizers,
then, denoting by $\varphi$ the function $x\mapsto x\log x$, the quantity $$\dfrac{\frac{1}{Z_t}\di \varphi(g^2)e^\frac{-V}{t}-\varphi\bigg(\frac{1}{Z_t} \di g^2e^\frac{-V}{t}\bigg)}{\frac{1}{Z_t}\di \lVert \nabla g \rVert^2e^\frac{-V}{t} }$$
can be easily estimated using Laplace's method, and is equal, for a generic function $g$, to $A+B\sqrt{t}+o(\sqrt{t})$ where $A$ and $B$ are constants that depend on the values of $g$ inside $\mathrm{argmin}(\om)$ and on $\po$.
Consequently, if we had any potentially quantitative $\Gamma$-convergence or any argument that allows to compute the quantity above for a single function $g$ (not depending on $t$) that would achieve equality in the log-Sobolev inequality satisfied by $\mu_0$, and letting $t \to 0$, we may recover the expansion of $C_{\mathrm{LS}}(\mu_t)$ as $C_{\mathrm{LS}}(\mu_0)+\alpha\sqrt{t}+o(\sqrt{t})$ using Laplace's method. Of course, optimal functions for the log-Sobolev inequality do not always exist, and despite $\Gamma$-convergence results existing \cite{Mariani2018gamma} in a more general setting, they are not enough to derive the desired conclusion. Yet, this heuristic suggests the asymptotics above.

\vspace{10pt}
In dimension $1$, the exact Poincaré and LSI constants for uniform measures on bounded intervals are known, in addition to optimal functions in the Poincaré inequality. It is also natural to begin with a square distance function as a PŁ potential. In this context, the proposition below shows that $C_{\mathrm{LS}}(\mu_t)-C_{\mathrm{LS}}(\mu_0)$ can be of order at least $\sqrt{t}$ (Theorem \ref{lsi1d} shows that it is exactly of order $\sqrt{t}$).
\begin{prop}\label{counterexple}
Let $f \colon  x \mapsto \frac{1}{2}\mathrm{dist}(x,[-\frac{\pi}{2},\frac{\pi}{2}])^2$, $\mu_t$ be the probability measure whose density is $\dfrac{1}{Z_t}e^{-\frac{f}{t}}$ with $Z_t$ enforcing unit mass and $\mu_0$ be the uniform measure on $[-\frac{\pi}{2},\frac{\pi}{2}]$. It holds that

\begin{equation}
C_{\mathrm{LS}}(\mu_t) \geq C_{\mathrm{LS}}(\mu_0) +\sqrt{\frac{8t}{\pi}} +o(\sqrt{t}).
\nonumber
\end{equation}
\end{prop}

\begin{proof}

Laplace's method yields for any continuous function $g$ having reasonable growth that
$$ \displaystyle\int_\mathbb{R} g(x) e^{\frac{-f(x)}{t}}\,\dd x = \displaystyle\int_{-\frac{\pi}{2}}^\frac{\pi}{2} g(x)\,\dd x +\big(g(\frac{\pi}{2})+g(-\frac{\pi}{2})\big)\sqrt{\frac{\pi t}{2}} + o(\sqrt{t}).$$
Therefore,
$$ \displaystyle\int_\mathbb{R} \sin^2(x) e^{\frac{-f(x)}{t}}\,\dd x = \frac{\pi}{2} +\sqrt{2\pi t} + o(\sqrt{t}),\quad \displaystyle\int_\mathbb{R} \cos^2(x) e^{\frac{-f(x)}{t}}\,\dd x = \frac{\pi}{2}+ o(\sqrt{t}).$$
By definition of the Poincaré constant, 
    $\displaystyle\int \sin^2(x)\,\dd \mu_t(x) \leq C_{\mathrm{P}}(\mu_t)\displaystyle\int \cos^2(x)\,\dd \mu_t(x).$
Thus,
$$C_{\mathrm{P}}(\mu_t) \geq 1 + \sqrt{\frac{8t}{\pi}} +o(\sqrt{t}) .$$
Finally, recall that $C_{\mathrm{LS}}(\mu_0) = C_{\mathrm{P}}(\mu_0) =1$ (Proposition \ref{segment}) and that $C_{\mathrm{LS}}(\mu_t) \geq C_{\mathrm{P}}(\mu_t)$ for all $t\geq 0$, hence
\[C_{\mathrm{LS}}(\mu_t)\geq C_{\mathrm{P}}(\mu_t) \geq 1 +\sqrt{\frac{8t}{\pi}} +o(\sqrt{t}) = C_{\mathrm{LS}}(\mu_0) +\sqrt{\frac{8t}{\pi}}+o(\sqrt{t}).\]
\end{proof}

\begin{remark}
\normalfont
Although Proposition \ref{counterexple} is based on the non generic fact that is $C_{\mathrm{P}}=C_{\mathrm{LS}}$ for uniform measures on intervals, the ideas presented earlier still suggest that the lower bound of order $\sqrt{t}$ remains true in higher dimensions and in a more general setting.
\end{remark}

We denote by $C_{\mathrm{P}}(t) \in [0,+\infty]$ the optimal constant in the Poincaré inequality satisfied by $\mu_t=\frac{1}{Z_t}e^\frac{-V}{t}$, where $V$ satisfies Assumption \ref{assumptions}.
\newcommand{\pp}{\frac{\pi}{2}}

Before detailing the full proof, we shall begin with a one-dimensional special case. Despite being somewhat restrictive, it captures the main ideas of the general proof while avoiding additional complications such as regularity issues that arise in higher dimensions.

\subsection{The one-dimensional case}\label{dim1poincaré}
In this subsection, we restrict our study to a one-dimensional special case. Up to rescaling and centering $\Omega$ (which will be a bounded interval in our case), we assume that

\begin{equation} 
  V(x) = \begin{cases}
 0  & \text{for $x \in [-\pp,\pp]$} \\
 \frac{\alpha}{2}(x-\pp)^2 & \text{for $x > \pp$} \\
 \frac{\beta}{2}(x+\pp)^2 & \text{for $x < -\pp$}
\end{cases}
\nonumber
\end{equation}

for some $\alpha, \beta >0$.

\begin{remark}
\normalfont
    The Lyapunov method applied in \cite{chewi2024ballisticlimitlogsobolevconstant} in addition to Holley-Stroock-type stability results \cite{bakry2014analysis} suggest that only the values of $V$ near $\mathrm{argmin}~V$ matter for the sought asymptotics. Therefore, applying this machinery to the example above should be sufficient to directly derive the result of Theorem \ref{1d} for all one-dimensional potentials satisfying Assumption \ref{assumptions} with $\alpha=2$.
    \end{remark}

\begin{theorem}\label{1d}
    Under the previous assumptions, the Poincaré constant satisfies $$C_{\mathrm{P}}(t)=1+\Big(\frac{1}{\sqrt{\alpha}}+\frac{1}{\sqrt{\beta}}\Big)\sqrt{\frac{2t}{\pi}} +o(\sqrt{t}).$$
\end{theorem}

The key argument to prove Theorem \ref{1d} is to show that we can replace functions (almost) achieving equality in the Poincaré inequality satisfied by $\mu_t$ with functions (almost) achieving equality in the Poincaré inequality satisfied by $\mu_0$ without significantly changing the ratio $\frac{\mathrm{Var}_{\mu_t}(.)}{\int \lvert \nabla .\rvert ^2\,\dd \mu_t}$ (up to slightly modifying the measure $\mu_t)$.

Before beginning the proof, we establish the following inequality.
\begin{lemma}\label{magic}
    Let $\varepsilon>0$. There exists a constant $C_\varepsilon>0$ such that the following inequality holds for all $g \in C^1(\mathbb{R}_+)$
    $$ \bigg\lvert \displaystyle\int_0^{+\infty}g(x)^2e^{-\frac{x^2}{2t}}\,\dd x-\sqrt{\frac{\pi t}{2}}g(0)^2 \bigg\rvert\leq \varepsilon \sqrt{t}g(0)^2+C_\varepsilon t\displaystyle\int_0^{+\infty} g'(x)^2 e^{-\frac{x^2}{2t}}\,\dd x$$
\end{lemma}

\begin{proof}
We shall prove the upper bound.
    Let $h=g-g(0)$ so that $h(0)=0$ and
\begin{equation}
\begin{split}
\displaystyle\int_0^{+\infty}g(x)^2e^{-\frac{x^2}{2t}}\,\dd x &=
\sqrt{\frac{\pi t}{2}} g(0)^2 + 2\displaystyle\int_0^{+\infty}g(0)h(x)e^{-\frac{x^2}{2t}}\,\dd x +\displaystyle\int_0^{+\infty}h(x)^2e^{-\frac{x^2}{2t}}\,\dd x \\
\end{split}
\nonumber
\end{equation}

Moreover,  $$2\displaystyle\int_0^{+\infty}g(0)h(x)e^{-\frac{x^2}{2t}}\,\dd x \leq  \displaystyle\int_0^{+\infty}\varepsilon g(0)^2e^{-\frac{x^2}{2t}}+ \varepsilon^{-1} h(x)^2e^{-\frac{x^2}{2t}}\,\dd x.$$

Therefore, 

\begin{equation}
\begin{split}
\displaystyle\int_0^{+\infty}g(x)^2e^{-\frac{x^2}{2t}}\,\dd x &\leq 
\sqrt{\frac{\pi t}{2} (1+\varepsilon)} g(0)^2  + (\varepsilon^{-1} +1)\displaystyle\int_0^{+\infty}h(x)^2e^{-\frac{x^2}{2t}}\,\dd x \\
&\leq \sqrt{\frac{\pi t}{2} (1+\varepsilon)} g(0)^2  + (\varepsilon^{-1} +1)\frac{t}{2}\displaystyle\int_0^{+\infty}h'(x)^2e^{-\frac{x^2}{2t}}\,\dd x\\
& = \sqrt{\frac{\pi t}{2} (1+\varepsilon)} g(0)^2  + (\varepsilon^{-1} +1)\frac{t}{2}\displaystyle\int_0^{+\infty}g'(x)^2e^{-\frac{x^2}{2t}}\,\dd x
\end{split}
\nonumber
\end{equation}
which is indeed what we wanted to prove, and where in the last inequality we applied the Gaussian Poincaré inequality to the function $\tilde{h}$ extending $h$ to $\mathbb{R}_-$ with $\tilde{h}(-x)=-h(-x)$ for $x \leq 0$.

The lower bound is obtained in the same way.
\end{proof}

\begin{remark}\label{generalizationmeancontrol}
    \normalfont The proof above shows that if $V$ is a symmetric potential and the measure $\eta= e^{-V}$ satisfies a Poincaré inequality with constant $C_{\mathrm{P}}(\eta)$, then the following inequality holds for all compactly supported smooth functions $g$ and all $\varepsilon >0$: $$\Bigg\lvert\di_0^{+\infty} (g^2-g(0)^2)e^{-V}\Bigg\rvert \leq \varepsilon \eta(\mathbb{R}_+)g(0)^2+(1+\dfrac{1}{\varepsilon})C_{\mathrm{P}}(\eta)\di_0^{+\infty}\lvert g'\rvert^2 e^{-V}.$$
\end{remark}

\begin{proofof}{Theorem \ref{1d}}

We divide the proof in 2 steps.
 For $t>0$ we denote by $\nu_t$ the probability measure defined by
    $$\di f \,\dd\nu_t = \frac{1}{Z_t}\Bigg( \di_{-\pp}^\pp f(x)\,\dd x+f(\pp)\sqrt{\pp\frac{t}{\alpha}}+f(-\pp)\sqrt{\pp\frac{t}{\beta}}~\Bigg)$$
for all continuous compactly supported functions $f \colon  \mathbb{R} \to \mathbb{R}$. Alternatively, $\nu_t$ is the push-forward measure of $\mu_t$ by the function $\mathrm{proj}_{[-\pp,\pp]}$, that is the usual projection on the convex set $[-\pp,\pp]$. We introduce 
\begin{equation}\label{sup}
C_t = \underset{g}{\sup} \dfrac{\di g^2 \,\dd\nu_t}{\frac{1}{Z_t}\displaystyle\int^\pp_{-\pp} g'(x)^2 \,\dd x} = \underset{g}{\sup} \dfrac{\di^\pp_{-\pp} g^2(x) \,\dd x+
\sqrt{\frac{\pi t}{2}
}\bigg(\frac{g(\pp)^2}{\sqrt{\alpha}} + \frac{g(-\pp)^2}{\sqrt\beta}\bigg)}{\displaystyle\int^\pp_{-\pp} g'(x)^2 \,\dd x}
\end{equation}
where the supremum is taken over all functions $g \in C^1([-\pp, \pp])$ such that $\di_{-\pp}^\pp g(x)\,\dd x=0$.

We will be first showing that $C_t=C_{\mathrm{P}}(t)+o(\sqrt{t})$ and then an estimate for $C_t$ can be obtained.

\vspace{10pt}
\noindent\textbf{Step 1.} \textit{$C_t=C_{\mathrm{P}}(t)+o(\sqrt{t})$.}

\vspace{5pt}
Clearly, $C_t \leq C_{\mathrm{P}}(t) +o(\sqrt{t})$. In fact, take $f \in C^1({[-\pp,\pp]})$ such that $\di_{-\pp}^\pp f(x)\,\dd x=0.$
We denote by $\tilde{f}$ the piecewise $C^1$ function extending $f$ to $\mathbb{R}$ such that $\tilde{f}$ is constant on $[\pp,+\infty)$ and $(-\infty,-\pp]$.
By the Poincaré inequality satisfied by $\mu_t$, 
\begin{equation}
\begin{split}
    \di f^2\,\dd\nu_t &= \di \tilde{f}^2\,\dd \mu_t
    \leq \bigg(\di \tilde{f}\,\dd \mu_t\bigg)^2 + C_{\mathrm{P}}(t)\frac{1}{Z_t}\di^\pp_{-\pp} f'(x)^2 \,\dd x \\
    & \leq  O(t)\di^\pp_{-\pp} f'(x)^2 \,\dd x + C_{\mathrm{P}}(t)\frac{1}{Z_t}\di^\pp_{-\pp} f'(x)^2 \,\dd x \text{~~~$\Big($as $\mu_t([-\pp,\pp]^c)=O(\sqrt{t}) \Big)$}
\end{split}
\nonumber
\end{equation}
where in the last inequality we used $\lVert f\rVert_{\infty,[-\pp,\pp]} \leq \sqrt{\pi\di^\pp_{-\pp} f'(x)^2\,\dd x}$ which follows from the Cauchy-Schwarz inequality and the fact that $f$ vanishes somewhere in $[-\pp,\pp]$. This also justifies that $C_t$ is well-defined and finite.\\
\par We now show the converse inequality.
For $t >0$, let $g_t \in C^{\infty}_c(\mathbb{R})$ be such that $$\dfrac{\mathrm{Var}_{\mu_t}(g_t)}{\displaystyle\int g_t'^2\,\dd \mu_t} \geq C_{\mathrm{P}}(t)-t.$$

Up to adding a constant to $g_t$ and multiplying $g_t$ by a positive factor, we may suppose $\displaystyle\int_{[-\pp, \pp]} g_t(x)\,\dd x=0$. By Lemma \ref{magic},

\allowdisplaybreaks
\begin{align*}
&\dfrac{\mathrm{Var}_{\mu_t}(g_t)}{\displaystyle\int g_t'^2\,\dd \mu_t}
\leq \dfrac{\di^\pp_{-\pp} g_t(x)^2\,\dd x+\di_\pp^{+\infty}g_t(x)^2 e^{\frac{-\alpha (x-\pp)^2}{2t}}\,\dd x+\di_{-\infty}^{-\pp}g_t(x)^2 e^{\frac{-\beta (x-\pp)^2}{2t}}\,\dd x}{\di^\pp_{-\pp} g_t'(x)^2\,\dd x + \di_\pp^{+\infty}g_t'(x)^2 e^{\frac{-\alpha (x-\pp)^2}{2t}}\,\dd x+\di_{-\infty}^{-\pp}g_t'(x)^2 e^{\frac{-\beta (x-\pp)^2}{2t}}\,\dd x}\\
&\leq \dfrac{\di^\pp_{-\pp} g_t(x)^2 \,\dd x +(1+\varepsilon)\sqrt{\frac{\pi t}{2}
}\bigg(\frac{g_t(\pp)^2}{\sqrt{\alpha}} + \frac{g_t(-\pp)^2}{\sqrt\beta}\bigg)}{\di^\pp_{-\pp} g_t'(x)^2\,\dd x + \di_\pp^{+\infty}g_t'(x)^2 e^{\frac{-\alpha (x-\pp)^2}{2t}}\,\dd x+\di_{-\infty}^{-\pp}g_t'(x)^2 e^{\frac{-\beta (x-\pp)^2}{2t}}\,\dd x}\\
&+ \dfrac{C_\varepsilon t \di_\pp^{+\infty}g_t'(x)^2 e^{\frac{-\alpha (x-\pp)^2}{2t}}\,\dd x+C_\varepsilon t \di_{-\infty}^{-\pp}g_t(x)^2 e^{\frac{-\beta (x-\pp)^2}{2t}}\,\dd x}{\di^\pp_{-\pp} g_t'(x)^2\,\dd x + \di_\pp^{+\infty}g_t'(x)^2 e^{\frac{-\alpha (x-\pp)^2}{2t}}\,\dd x+\di_{-\infty}^{-\pp}g_t'(x)^2 e^{\frac{-\beta (x-\pp)^2}{2t}}\,\dd x}
\\
&\leq \dfrac{\di^\pp_{-\pp} g_t(x)^2\,\dd x +(1+\varepsilon)\sqrt{\frac{\pi t}{2}
}\bigg(\frac{g_t(\pp)^2}{\sqrt{\alpha}} + \frac{g_t(-\pp)^2}{\sqrt\beta}\bigg)}{\di^\pp_{-\pp} g_t'(x)^2\,\dd x}
\end{align*}
\normalsize for sufficiently small values of $t$.

Note that in the last inequality, we have used the fact that $\frac{a+b}{c+d}\leq \frac{a}{c}$ whenever $a,b,c,d$ are positive real numbers satisfying $\frac{a}{c}\geq\frac{b}{d}$.
Therefore, for $t$ sufficiently small,
$$\dfrac{\mathrm{Var}_{\mu_t}(g_t)}{\displaystyle\int g_t'^2\,\dd \mu_t} \leq \dfrac{\di g_t^2 \,\dd\nu_t}{\frac{1}{Z_t}\displaystyle\int g_t'(x)^2 \,\dd x}+O(\varepsilon\sqrt{t})$$
since $\big(\frac{g_t(\pp)^2}{\sqrt{\alpha}} + \frac{g_t(-\pp)^2}{\sqrt\beta}\big) \big{/} \displaystyle\int^\pp_{-\pp} g_t'(x)^2\,\dd x$ is uniformly bounded as $\di_{-\pp}^\pp g_t(x)\,\dd x=0$.

Letting $\varepsilon \to 0$ yields that $C_{\mathrm{P}}(t) \leq C_t +o(\sqrt{t})$.

\vspace{5pt}
\noindent\textbf{Step 2.} \textit{Estimating $C_t$.}

\vspace{10pt}
Now that we have shown that $C_{\mathrm{P}}(t)=C_t+o(\sqrt{t})$, we can compute the asymptotics of $C_t$ using standard techniques.

From Section \ref{prelim}, we have that $C_0=C_{\mathrm{P}}(0)=1$ and that equality is achieved for the function $g_0=\sin$. We shall start with a lower bound :

\begin{equation}
    \begin{split}
C_t &\geq \dfrac{\di g_0^2 \,\dd\nu_t}{\frac{1}{Z_t}\displaystyle\int^\pp_{-\pp} g_0'(x)^2 \,\dd x} = \dfrac{\di^\pp_{-\pp} \sin(x)^2 \,\dd x+
\sqrt{\frac{\pi t}{2}
}\bigg(\frac{\sin(\pp)^2}{\sqrt{\alpha}} + \frac{\sin(-\pp)^2}{\sqrt\beta}\bigg)}{\displaystyle\int^\pp_{-\pp} \cos(x)^2 \,\dd x}\\
& = 1+ \dfrac{(\frac{1}{\sqrt{\alpha}}+\frac{1}{\sqrt{\beta}})\sqrt{\frac{\pi t}{2}}}{\pp}= 1+\big(\frac{1}{\sqrt{\alpha}}+\frac{1}{\sqrt{\beta}}\big)\sqrt{\frac{2t}{\pi}}.
    \end{split}
    \nonumber
\end{equation}

Let us now prove a similar upper bound.
For $t>0$, let $g_t$ be a function satisfying the same conditions as above, i.e. $\di^\pp_{-\pp}g_t(x)\,\dd x=0$ and
$$C_t+o(\sqrt{t}) = \dfrac{\di g_t^2 \,\dd\nu_t}{\frac{1}{Z_t}\displaystyle\int^\pp_{-\pp} g_t'(x)^2 \,\dd x} $$

By homogeneity, we may suppose that $\displaystyle\int^\pp_{-\pp} g_t'(x)^2\,\dd x=1$.
We can therefore take (up to a subsequence) a weak $H^1$ limit of $g_t$ as $t\to 0$ that we call $u$. By Ascoli theorem, we may suppose without loss of generality that the convergence holds also with respect to the $L^{\infty}$ norm (and hence that $u$ is continuous). Hence
$$1=C_0 \geq \dfrac{\di^\pp_{-\pp}u(x)^2\,\dd x}{\di^\pp_{-\pp}u'(x)^2\,\dd x} \geq \limsup  \dfrac{\di g_t^2 \,\dd\nu_t}{\frac{1}{Z_t}\displaystyle\int^\pp_{-\pp} g_t'(x)^2 \,\dd x}=1$$
Therefore, $\di^\pp_{-\pp}u(x)\,\dd x=0$ and $\di^\pp_{-\pp}u(x)^2 \,\dd x=\di^\pp_{-\pp}u'(x)^2 \,\dd x=1$. By uniqueness of equality cases in \eqref{sup} for $t=0$\footnote{There is, up to a scalar factor and an additive constant, a unique optimal function for Poincaré inequality on a segment, namely $\sin$ when the segment is $[-\pp, \pp]$. In fact, if $u \in H^1([-\pp, \pp])$ achieves equality and $\int_{[-\pp,\pp]}u=0$, then it satisfies an Euler-Lagrange equation $\int_{[-\pp,\pp]} uh = C \int_{[-\pp,\pp]} u'h'$ for all test functions $h \in C^{\infty}_c([-\pp,\pp])$ and therefore, $h$ is a sine function.}, $u= \lambda \sin$ for $\lambda = \pm \dfrac{1}{\pi}$.
It follows that the convergence holds for $t\to 0$ (up to changing the sign of $g_t$ for every $t$) and that
\allowdisplaybreaks
\begin{align*}
    C_t+o(\sqrt{t}) &= \dfrac{\di^\pp_{-\pp} g_t(x)^2\,\dd x}{\di^\pp_{-\pp} g_t'(x)^2\,\dd x}+\dfrac{\sqrt{\frac{\pi t}{2}}\bigg(\frac{g_t(\pp)^2}{\sqrt{\alpha}} + \frac{g_t(-\pp)^2}{\sqrt\beta}\bigg)}{\di^\pp_{-\pp} g_t'(x)^2\,\dd x}\\
    &\leq 1 +\Big(\frac{1}{\sqrt{\alpha}}+\frac{1}{\sqrt{\beta}}\Big)\sqrt{\frac{\pi t}{2}}\Bigg(\dfrac{g_t(\pp)^2+ g_t(-\pp)^2}{\di^\pp_{-\pp} g_t'(x)^2\,\dd x}\Bigg)\\
    &= 1+\Big(\frac{1}{\sqrt{\alpha}}+\frac{1}{\sqrt{\beta}}\Big)\sqrt{\frac{\pi t}{2}}\Bigg(\dfrac{u(\pp)^2+ u(-\pp)^2}{\di^\pp_{-\pp} u'(x)^2\,\dd x}+o(1)\Bigg)\\
    & =  1+\big(\frac{1}{\sqrt{\alpha}}+\frac{1}{\sqrt{\beta}}\big)\sqrt{\frac{2t}{\pi}} +o(\sqrt{t}).
    \end{align*}
which concludes the proof.
\end{proofof}

\begin{remarks}
\normalfont
\begin{enumerate}
\item One can show that $C_t$ (up to a $o(\sqrt{t})$ error) is also equal to $$\underset{g}{\sup} \dfrac{\text{Var}_{\nu_t}(g)}{\frac{1}{Z_t}\displaystyle\int^\pp_{-\pp} g'(x)^2 \,\dd x},$$ but in this case, an Euler-Lagrange equation for an optimal function $u$ such that $\int u \,\dd\nu_t=0$ is (like in the case $t=0$) $\int_{[-\pp,\pp]} uh = C_t \int_{[-\pp,\pp]} u'h'$ for $h \in C^{\infty}_c\big((-\pp,\pp)\big)$. Therefore $u$ is a sine with frequency $\frac{1}{\sqrt{C_t}}$ but it is not necessarily centered and one can perform tedious (but feasible) calculations to re-derive Theorem $\ref{1d}$.

\item Optimal functions always exist by the $H^1$ weak compactness argument presented in the previous proof.
\item Although applicable and relevant in the one-dimensional case, the Ascoli Theorem may not be the right approach : in fact, the right generalization in higher dimensions (as we lose the embedding $H^1 \to C^0$) would be the compactness of the embedding $H^1 \to L^2$ and the existence of a continous trace ``embedding'' $H^1([-\pp,\pp]) \to L^2(\partial [-\pp,\pp])$.

\end{enumerate}
\end{remarks}

\subsection{The general case}
Recall Assumption \ref{assumptions}. We denote by $\oc$ the open domain $\mathbb{R}^n\backslash \bar\om$.
Since $V$ is convex, $\om$ is convex too. The boundary $\po$ being smooth, it follows that $(\theta,r) \in\po \times\mathbb{R}_+ \mapsto \theta+r\nu(\theta)$ is a diffeomorphism from $\po \times\mathbb{R}_+$ onto $\mathbb{R}^n\backslash\om$, where $\nu(\theta)$ is the outer unit normal vector.

\vspace{5pt}
\noindent In the sequel, for all $f \colon  \oc \to \mathbb{R}$, we will denote by $f(r,\theta)$ the image of $\theta+r\nu(\theta)$ under $f$. Besides, for all $\varphi \in C_c(\oc)$, it holds that
$$\dic \varphi(x)\,\dd x=\di_{\theta \in \po}\di_{r\geq 0} \varphi(r,\theta) \kappa(r,\theta)\, \dd rd\theta$$
where $d\theta$ is the surface measure on $\po$ and $\kappa(r,\theta)$ is the Jacobian associated to the change of variables.
An explicit formula for $\kappa(r,\theta)$ is given by $\displaystyle\prod_{i}(1+r\kappa_i)$ where $\kappa_1,\dots,\kappa_{n-1}$ are the principal curvatures of $\po$ at $\theta$. Note that since $\om$ is convex, $\kappa_i \geq 0$ for all $i$. An important consequence of this is that, for all $\theta \in \po$, the map $\kappa(.,\theta)$ is log-concave.

\begin{remark}
    \normalfont
    A possible strategy to study the general case is to adapt the one-dimensional argument of Theorem \ref{1d} in the following way :
    \begin{itemize}
        \item Define 
$$C_t = \underset{g}{\sup} \dfrac{\di_\om g(x)^2\,\dd x+\di_{\oc}g^t(x)^2e^\frac{-V(x)}{t}\,\dd x}{\di_{\om} \lVert\nabla g(x)\rVert^2\,\dd x+\di_{\oc}\lVert\nabla g^t(x)\rVert^2 e^\frac{-V(x)}{t}\,\dd x} $$
where the supremum is taken over all functions $g \in H^1(\om)$ such that $\di_\om g(x)\,\dd x=0$ and where, for $f \in H^1(\Omega)$, $f^t \in H^1(\om,\,\dd \mu_{t | \oc})$ is the unique minimizer of $$h \mapsto \di_{\oc}  \lVert\nabla h(x)\rVert^2 e^\frac{-V(x)}{t}\,\dd x$$ such that $f^t_{|\partial\om}=f_{|\partial\om}$.
\item Show that $C_{\mathrm{P}}(t)=C_t+o(Z_t-\lvert\om\rvert)$ using inequalities generalizing Lemma \ref{magic}.
\item Compute the asymptotics of $C_t$.
    \end{itemize}
This last step is the most difficult because of the term $\di_{\oc}\lVert\nabla g^t(x)\rVert^2 e^\frac{-V(x)}{t}\,\dd x$, which is straightforward to estimate only in dimension $1$. Therefore, we need another, slightly different, approach.
\end{remark}

We begin with estimating $Z_t-\om$.

\begin{lemma}\label{zt}
    Under Assumption \ref{assumptions}, $Z_t-\lvert\om\rvert \sim \gamma t^\frac{1}{\alpha}$ where $\gamma$ is a positive constant.
\end{lemma}
\begin{proof}
By Assumption \ref{assumptions}, we may write $V(r,\theta)=\dfrac{r^\alpha}{a(\theta)+\eta(r,\theta)}$ where $\eta(r,\theta) \xrightarrow[r \to 0^+]{} 0$ uniformly in $\theta$. Therefore,
    \begin{align*}
        \begin{split}
            Z_t-\lvert \om\rvert 
            &= \di_{\oc} e^\frac{-V}{t}=\di_{\po}\di_{r\geq 0} e^\frac{-V(r,\theta)}{t}\kappa(r,\theta)\,\dd r\dd\theta= \di_{\po}\di_{r\geq 0} e^\frac{-r^\alpha}{t(a(\theta)+\eta(r,\theta))}\kappa(r,\theta)\,\dd r\dd\theta\\
            &=t^\frac{1}{\alpha}\di_{\po}\di_{u\geq 0} e^{-(a(\theta)+\eta(t^\frac{1}{\alpha}u,\theta))^{-1}u^\alpha}\kappa(t^\frac{1}{\alpha}u,\theta) \,\dd u\,\dd\theta ~~~~ (r\coloneqq t^\frac{1}{\alpha}u)\\
            &\geq t^\frac{1}{\alpha}\di_{\po}\di_{u\geq 0}\mathbf{1}_{a(\theta)>\delta} e^{-(a(\theta)+\eta(t^\frac{1}{\alpha}u,\theta))^{-1}u^\alpha}\kappa(t^\frac{1}{\alpha}u,\theta) \,\dd u\,\dd\theta
        \end{split}
    \end{align*}
where $\delta$ is any positive real number. Hence, by dominated convergence,
$$\lim\inf \dfrac{Z_t-\lvert \om \rvert}{t^\frac{1}{\alpha}} \geq \di_{\po}\di_{u\geq 0} \mathbf{1}_{a(\theta)>\delta}e^{-a(\theta)^{-1}u^\alpha} \,\dd u\,\dd\theta.$$
Letting $\delta \to 0$ yields by monotone convergence
$$\lim\inf \dfrac{Z_t-\lvert \om \rvert}{t^\frac{1}{\alpha}} \geq \di_{\po}\di_{u\geq 0} \mathbf{1}_{a(\theta)>0}e^{-a(\theta)^{-1}u^\alpha} \,\dd u\,\dd\theta.$$
For the upper bound, observe that $$\di_{\po}\di_{r\geq \xi} e^\frac{-V(r,\theta)}{t}\kappa(r,\theta)\,\dd r\dd\theta = o(t^\frac{1}{\alpha})$$
as the above integral is exponentially small, so that for $\delta>0$ and $\xi>0$ such that $\lvert\eta(r,\theta)\rvert \leq \frac{\delta}{2}$ for all $\theta \in \po$ and $r \leq \xi$,

\allowdisplaybreaks
\begin{equation*}
        \begin{split}
            Z_t-\lvert \om\rvert 
            &= \di_{\oc} e^\frac{-V}{t}=\di_{\po}\di_{r\geq 0} e^\frac{-V(r,\theta)}{t}\kappa(r,\theta)\,\dd r\dd\theta=o_{\xi}(t^\frac{1}{\alpha})+\di_{\po}\di_{r\leq \xi} e^\frac{-V(r,\theta)}{t}\kappa(r,\theta)\,\dd r\dd\theta\\
            & \leq o_{\xi}(t^\frac{1}{\alpha})+\di_{\po}\di_{r\leq \xi} \mathbf{1}_{a(\theta)>\delta} e^\frac{-V(r,\theta)}{t}\kappa(r,\theta)\,\dd r\dd\theta+\di_{\po}\di_{r\leq \xi} \mathbf{1}_{a(\theta)\leq\delta}e^\frac{-V(r,\theta)}{t}\kappa(r,\theta)\,\dd r\dd\theta\\
            &\leq o_\xi(t^\frac{1}{\alpha})+t^\frac{1}{\alpha}\di_{\po}\di_{u\geq 0}\mathbf{1}_{a(\theta)>\delta} e^{-(a(\theta)+\eta(t^\frac{1}{\alpha}u,\theta))^{-1}u^\alpha}\kappa(t^\frac{1}{\alpha}u,\theta) \,\dd u\,\dd\theta\\&+\di_{\po}\di_{r\leq \xi} \mathbf{1}_{a(\theta)\leq\delta}e^\frac{2r^\alpha}{\delta t}\kappa(r,\theta)\,\dd r\,\dd\theta\\
            &= o_\xi(t^\frac{1}{\alpha})+t^\frac{1}{\alpha}\di_{\po}\di_{u\geq 0}\mathbf{1}_{a(\theta)>\delta} e^{-(a(\theta)+\eta(t^\frac{1}{\alpha}u,\theta))^{-1}u^\alpha}\kappa(t^\frac{1}{\alpha}u,\theta) \,\dd u\,\dd\theta+O((\delta t)^\frac{1}{\alpha})\\
        \end{split}
    \end{equation*}
yielding
$$\lim\sup \dfrac{Z_t-\lvert \om \rvert}{t^\frac{1}{\alpha}} \leq \di_{\po}\di_{u\geq 0} \mathbf{1}_{a(\theta)>0}e^{-a(\theta)^{-1}u^\alpha} \,\dd u\,\dd\theta +O(\delta^\frac{1}{\alpha})$$
from which it follows that $$\lim\sup \dfrac{Z_t-\lvert \om \rvert}{t^\frac{1}{\alpha}} \leq \di_{\po}\di_{u\geq 0} \mathbf{1}_{a(\theta)>0}e^{-a(\theta)^{-1}u^\alpha} \,\dd u\,\dd\theta \leq  \lim\inf \dfrac{Z_t-\lvert \om \rvert}{t^\frac{1}{\alpha}}.$$
Since $\theta \mapsto a(\theta)$ is bounded, we have $V(r,\theta) \geq cr^\alpha$ for all $r>0$, where $c$ is a positive constant. Therefore,
\begin{align*}
        \begin{split}
            Z_t-\lvert \om\rvert = \di_{\oc} e^\frac{-V}{t}
            =\di_{\po}\di_{r\geq 0} e^\frac{-V(r,\theta)}{t}\kappa(r,\theta)\,\dd r\,\dd\theta\leq \di_{\po}\di_{r\geq 0} e^\frac{-cr^\alpha}{t}\kappa(r,\theta)\,\dd r\,\dd\theta= O(t^\frac{1}{\alpha})
        \end{split}
    \end{align*}
and
$$\lim \dfrac{Z_t-\lvert \om \rvert}{t^\frac{1}{\alpha}} =\di_{\po}\di_{u\geq 0} \mathbf{1}_{a(\theta)>0}e^{-a(\theta)^{-1}u^\alpha} \,\dd u\dd\theta=\di_{\po}a(\theta)^\frac{1}{\alpha} \,\dd \theta \di_{u\geq 0} e^{-u^\alpha}\,\dd u.$$
which completes the proof.
\end{proof}

The following proposition generalizes Lemma \ref{magic}.
\begin{proposition}\label{klslemma}
    Let $U,W \in C(\mathbb{R}_+)$ be convex potentials such that $c_1x^\alpha \leq U(x) \leq c_2x^\beta$ on a neighborhood of $0$, where $c_1,c_2,\alpha,\beta$ are positive constants. Then, for $h \in C^1(\R_+)$ and small $\varepsilon>0$,
    $$\left\vert\di_{\R_+}\left(h^2-h^2(0)\right)e^{-\frac{U}{t}-W}\right\vert \leq \varepsilon h(0)^2+ \dfrac{C}{\varepsilon}t^{\frac{3}{\alpha}-\frac{1}{\beta}}\di_{\R_+}h'^2e^{-\frac{U}{t}-W}$$
    for some positive constant $C$ (depending on $U$ and $W$).
\end{proposition}
\begin{proof} 
    Proposition \ref{klslemma} is a particular case of Remark \ref{generalizationmeancontrol}. It is enough to show that $$C_P(e^{-\frac{U}{t}-W}) \leq C_{U}t^{\frac{3}{\alpha}-\frac{1}{\beta}}$$ for some constant $C_U$ that may depend on $U$.
    This can be seen as a consequence of Proposition \ref{kls1d}. Indeed,
    \begin{align*}
        \begin{split}
        C_P(e^{-\frac{U}{t}}) &\leq \dfrac{\di_0^\infty x^2 e^{-\frac{U(x)}{t}-W(x)}\,\dd x}{\di_0^\infty e^{-\frac{U(x)}{t}-W(x)}\,\dd x}\leq \dfrac{\di_0^\infty x^2 e^{-\frac{c_1x^\alpha}{t}-W(x)}\,\dd x}{\di_0^\infty e^{-\frac{c_2x^\beta}{t}-W(x)}\,\dd x} = t^{\frac{3}{\alpha}-\frac{1}{\beta}}  \dfrac{\di_0^\infty y^2 e^{-c_1y^\alpha-W(t^\frac{1}{\alpha} y)}\,\dd x}{\di_0^\infty e^{-c_2y^\beta-W(t^\frac{1}{\beta} y)}\,\dd x}.
        \end{split}
    \end{align*}
Hence $ C_P(e^{-\frac{U}{t}}) =O(t^{\frac{3}{\alpha}-\frac{1}{\beta}})$ since $W$ is bounded on a neighborhood of $0$.
\end{proof}

Lemma \ref{klslemma} is designed to be applied to half-lines pointing outwards from the domain $\om$. In the following lemma, it allows us to estimate integrals over $\oc$ by integrals with respect to a limiting surface measure that arises in the limit $t \to 0$, which captures the growth of $V$ near the boundary $\po$.

\begin{lemma}\label{magicnd}
Define $p$ as the orthogonal projection onto the convex set $\om$ and let $\sigma_t$ be the push-forward of the probability measure $\frac{\mathbf{1}_{\oc}}{Z_t-\om}\mu_t$ by $p$, so that for all $u \in C(\partial \om)$, it holds that
\begin{align*}
    \di_{\partial \om} u \,\dd{\sigma_t}=\dfrac{1}{Z_t-|\om|}\di_{\po}\di_{r\geq 0} u(\theta)\kappa(r,\theta)e^\frac{-V(r,\theta)}{t} \,\dd r\dd\theta.
    \end{align*}
Then, for all $h \in H^1(\oc,\,\dd \mu_t)$, we have 
$$\di_{\oc} h^2(x)e^\frac{-V(x)}{t}\,\dd x = \left(1+o(1)\left(1+\lVert h \rVert^2_{H^1(\om)}\right)\right)(Z_t-|\om|) \di_{\partial \om} h^2 \dd\tilde{\sigma}+o(Z_t-\lvert \om \rvert)\di_{\oc} \lVert\nabla h(x) \rVert ^2 e^\frac{-V(x)}{t}\,\dd x$$
where $ \tilde{\sigma}$ is the probability measure supported on $\om$ such that $$\dfrac{\,\dd \sigma_t}{\,\dd \theta} \xrightarrow[]{L^\infty(\dd\theta)} \dfrac{\,\dd\tilde{\sigma}}{\,\dd \theta} \propto a(\theta)^\frac{1}{\alpha}.$$
\end{lemma}

\begin{proof}
The beginning of the proof goes along the same lines as Lemma \ref{magic}.
Let $h \in C^{\infty}_c(\mathbb{R}^d)$, then, by Proposition \ref{klslemma},
\begin{align*}
    \begin{split}
    &\left\vert \dic h(x)^2e^\frac{-V(x)}{t} \,\dd x-(Z_t-\lvert \om \rvert)\di_{\partial\om} h^2 \,\dd \sigma_t  \right\vert
     \leq \di_{\theta \in \partial \om} \varepsilon h(\theta)^2 \di_{r}   e^\frac{-V(r,\theta)}{t}\kappa(r,\theta)\dd rd\theta\\&       +       \di_{\theta \in \partial \om}   \frac{C}{\varepsilon}t^{\frac{3}{\alpha}-\frac{1}{\beta}}\di_{r}   \Big(\frac{d}{\dd r}h(r,\theta)\Big)^2 e^\frac{-V(r,\theta)}{t}\kappa(r,\theta)\dd rd\theta\\&\leq \varepsilon (Z_t-\lvert \om \rvert)\di_{\partial\om} h^2 \,\dd \sigma_t +\frac{C}{\varepsilon}t^{\frac{3}{\alpha}-\frac{1}{\beta}}\di_{\theta \in \partial \om}  \di_{r}    \lVert \nabla h(r,\theta)\rVert^2 e^\frac{-V(r,\theta)}{t}\kappa(r,\theta)\dd rd\theta.\\
    \end{split}
\end{align*}
Moreover, let $\delta>0, k\in \mathbb{N}^*$ and let $\xi>0$ be such that $\underset{\theta \in \po,~r\leq \xi}{\sup}\lvert \eta(r,\xi)\rvert \leq \frac{\delta}{k}$. We have, for $\theta \in \po$ such that $a(\theta)>\delta$,
\begin{align*}
    \begin{split}
    \dfrac{\,\dd \sigma_t}{\,\dd \theta}&=\di_{r\geq 0}\frac{1}{Z_t-\lvert \om\rvert}\kappa(r,\theta)e^{\frac{-V(r,\theta)}{t}}\,\dd r =\frac{1}{Z_t-\lvert \om\rvert} \di_{r\geq 0} e^\frac{-r^\alpha}{t(a(\theta)+\eta(r,\theta))}\kappa(r,\theta)\,\dd r \\&= o_{\delta,k}(1)+\frac{1}{Z_t-\lvert \om\rvert}\di_{r\leq \xi}  e^\frac{-r^\alpha}{t(a(\theta)+\eta(r,\theta))}\kappa(r,\theta)\,\dd r= o_{\delta,k}(1)+\frac{1}{Z_t-\lvert \om\rvert}\di_{r\leq \xi}  e^\frac{-r^\alpha}{t(a(\theta)+\eta(r,\theta))}\kappa(r,\theta)\,\dd r\\
    &=\frac{\left(a(\theta)+O(\frac{\delta}{k})\right)^\frac{1}{\alpha}\left(o_{\delta,k}(1)+\di_{u\geq 0} e^{-u^\alpha}\right)}{\gamma+o(1)}.
    \end{split}
\end{align*}
Hence $\dfrac{\,\dd \sigma_t}{\,\dd \theta} \longrightarrow \dfrac{a(\theta)^\frac{1}{\alpha}}{\di_{\po} a(\theta)^\frac{1}{\alpha}\dd\theta} $ uniformly in $\{\theta \in \po : a(\theta)>\delta\}$.
Similarly, we have $$\underset{t \to 0^+}{\limsup}\, \underset{\theta}{\sup} \mathbf{1}_{a(\theta)\leq \delta }\dfrac{\,\dd \sigma_t}{\,\dd \theta} = \underset{\delta \to 0^+}{o}(1)$$
from which it follows that $$\dfrac{\,\dd \sigma_t}{\,\dd \theta} \xrightarrow[]{L^\infty(\po)}\dfrac{a(\theta)^\frac{1}{\alpha}}{\di_{\po} a(\theta)^\frac{1}{\alpha}\,\dd\theta}\coloneqq \dfrac{\,\dd \tilde{\sigma}}{\,\dd \theta}.$$
Thus, for $h \in C^\infty_c(\R)$, it holds that
\begin{align*}
    \begin{split}
        \di_{\po} h^2\,\dd\sigma_t=\di_{\po} h^2\,\dd\tilde{\sigma}+o(1)\di_{\po} h^2\,\dd\theta=\di_{\po} h^2\,\dd\tilde{\sigma}+o(1)\lVert h \rVert^2_{L^2(\dd \theta)}=\di_{\po} h^2\,\dd\tilde{\sigma}+o(1)\lVert h \rVert^2_{H^1(\om)}
    \end{split}
\end{align*}
Letting $\varepsilon=\log(t^{-1})^{-1}$ in
\begin{align*}
    \begin{split}
    \left\vert \dic h(x)^2e^\frac{-V(x)}{t} \,\dd x-(Z_t-\lvert \om \rvert)\di_{\partial\om} h^2 \,\dd \sigma_t  \right\vert
    \leq &\varepsilon (Z_t-\lvert \om \rvert)\di_{\partial\om} h^2 \,\dd \sigma_t\\&+\frac{C}{\varepsilon}t^{\frac{3}{\alpha}-\frac{1}{\beta}}\di_{\theta \in \partial \om}  \di_{r}    \lVert \nabla h(r,\theta)\rVert^2 e^\frac{-V(r,\theta)}{t}\kappa(r,\theta)\,\dd rd\theta
    \end{split}
\end{align*}
and observing that $t^{\frac{3}{\alpha}-\frac{1}{\beta}}\log(t^{-1})=o(t^{\frac{1}{\alpha}})$ yields the desired conclusion.
\end{proof}

Finally, we recall a regularity estimate.
\begin{lemma}\label{pasencore}
    Let $V \in C^{1}_{loc}(\mathbb{R}^n)$ be a nonnegative convex function such that $\di\!e^{-V}$ is finite. Let $L_V=\Delta-\nabla V.\nabla$ and $g \in H^2_{loc}(\mathbb{R}^n)\cap H^1(e^{-V})$ be such that 
    $L_Vg=-\lambda g$ for some constant $\lambda \in \mathbb{R}$. Then, the eigenfunction $g\in H^2(e^{-V})$ and $\di \lVert \nabla^2g\rVert_F^2e^{-V} \leq \lambda^2 \di g^2e^{-V} $.
\end{lemma}

\begin{proof}
Let $(V_n)_n$ be a sequence of smooth convex functions such that $V_n \to V$ in $W^{1,\infty}_{loc}(\mathbb{R}^n)$, and let $\varphi \in C^{\infty}_c(\mathbb{R}^n)$.
By the integrated Bochner formula (Proposition \ref{bochner}), it holds that
\begin{equation}
    \begin{split}
        \di (L_{V_n}\varphi)^2e^{-V_n}&=\di \lVert \nabla^2 \varphi \rVert_F^2e^{-V_n}+\di \langle\nabla^2V_n \nabla\varphi,\nabla \varphi\rangle e^{-V_n}\\
        &\geq \di \lVert \nabla^2 \varphi \rVert_F^2e^{-V_n}. 
    \end{split}
    \nonumber
\end{equation}

Taking the limit as $n \to \infty$, it follows that $\di (L_V\varphi)^2e^{-V} \geq \di \lVert \nabla^2 \varphi \rVert_F^2e^{-V} $.
Using a standard density argument, the latter inequality holds for all compactly supported functions $g \in H^2(\mathbb{R}^n)$.
Let $R>0$. For $n \in \mathbb{N}$, let $\chi_n \in C_c^{\infty}(\mathbb{R}^n)$ be a smooth cutoff function such that $\chi =1$ in the centered ball of radius $R$, and such that $ \lVert\nabla\chi_n\rVert_{\infty}+\lVert \nabla^2 \chi_n\rVert_\infty\to 0$. 
It follows that $L_V(\chi_ng)-\chi_nL_V(g)=gL_V(\chi_n) +2\nabla \chi_n.\nabla g \xrightarrow[]{L^2(e^{-V}\,\dd x)}0$. Hence, 
\begin{equation}
    \begin{split}
        \lambda^2\di g^2e^{-V} &= \di L_V(g)^2e^{-V} \geq \di\chi_n^2L_V(g)^2e^{-V}=\di L_V(\chi_n g)^2e^{-V}+o(1) \quad \text{as } n \to \infty\\
        &\geq \di \lVert \nabla^2 (\chi_ng)\rVert_F^2e^{-V}+o(1)\geq \di_{B_R} \lVert \nabla^2 g\rVert_F^2e^{-V}+o(1).
        \end{split}
    \nonumber
\end{equation}
The result follows by letting $R\to \infty$.
\end{proof}

We are now in position to state and prove the precise version of Theorem \ref{thm1}.
\begin{theorem}\label{mainthm}
    We have $$\dfrac{C_{\mathrm{P}}(t)-C_{\mathrm{P}}(0)}{Z_t -\lvert \om \rvert} \xrightarrow[t\to 0]{} \underset{f}{\sup} \bigg(\di_{\po}f^2\dd\tilde{\sigma}-C_{\mathrm{P}}(0)\di_{\po}\lVert \nabla f \rVert^2\dd\tilde{\sigma}\bigg) \in \mathbb{R}$$
    where the supremum above is taken over all Neumann eigenfunctions $f \colon  \om \to \mathbb{R}$ associated to the smallest eigenvalue and such that $\di_\om\lVert\nabla f(x) \rVert^2\,\dd x = 1$.
    
\end{theorem}

\begin{remark}
\normalfont
    Recalling that $Z_t- \lvert \om \rvert \sim \gamma t^\frac{1}{\alpha}$, Theorem \ref{mainthm} states that $$C_P(t)=C_P(0)+\Lambda_{\om, V}t^\frac{1}{\alpha}+o(t^\frac{1}{\alpha})$$
    where $$\Lambda_{\om, V}=C_\alpha\left(\di_{\po}a(\theta)^\frac{1}{\alpha} \,\dd \theta \right) ~ \underset{f}{\sup} \bigg(\di_{\po}f^2 \,\dd\tilde{\sigma}-C_{\mathrm{P}}(0)\di_{\po}\lVert \nabla f \rVert^2 \,\dd\tilde{\sigma}\bigg)$$ with
    $C_\alpha=\di_{u\geq 0} e^{-u^\alpha}\,\dd u=\alpha^{-1}\Gamma(\alpha^{-1})$ and where the supremum is taken over the same set as in Theorem $\ref{mainthm}$. Therefore, $\dfrac{\Lambda_{\om, V}}{C_\alpha \di_{\po}a(\theta)^\frac{1}{\alpha} \,\dd \theta }$ is itself the optimal constant $C$ in the functional inequality $\di_{\po}f^2 \,\dd\tilde{\sigma}\leq C_{\mathrm{P}}(0)\di_{\po}\lVert \nabla f \rVert^2 \,\dd\tilde{\sigma}+C\di_\om \lVert \nabla f \rVert^2$ satisfied by the Neumann eigenfunctions mentioned above.
\end{remark}

\begin{remark}\label{lam}
\normalfont
It is part of the theorem that $\Lambda_{\om,V}$ is finite. Moreover, considering the generic case where $\om$ is such that the smallest nonzero Neumann eigenvalue is simple, since $\tilde{\sigma}$ only depends on the growth of $V$ in $\oc$ near $\po$, which can be perturbed slightly without violating Assumption \ref{assumptions}, the constant $\Lambda_{\om,V}$ is expected to be generically nonzero as it is an integral depending on the Neumann eigenfunction with respect to the perturbed measure. In this context, Theorem \ref{1d} provides an example where $\Lambda_{\om, V} \neq 0$.
\end{remark}

\begin{proof}

For $t>0$, we denote by $g_t$ a function achieving equality in the Poincaré inequality satisfied by the measure $\mu_t$, such that $\di g_t\,\dd \mu_t=0$ and that $\di_\om \lVert \nabla g_t \rVert^2\,\dd \mu_t=1$. Therefore $$L_tg_t=-\dfrac{1}{C_{\mathrm{P}}(t)}g_t \text{ and } C_{\mathrm{P}}(t)=\dfrac{\di g_t^2\,\dd \mu_t}{\di \lVert\nabla g_t \rVert^2\,\dd \mu_t}$$ where $L_t=\Delta-\frac{1}{t}\nabla V.\nabla$.

Up to rescaling $\Omega$ and $V$, we may suppose that $C_{\mathrm{P}}(0)=1$. For simplicity, we suppose that the eigenvalue $1$ is simple for the Laplace operator under Neumann boundary conditions on $\om$ (see footnote $4$ for a clarification).

We divide the proof in three steps :

\vspace{10pt}
\noindent\textbf{Step 1.} \textit{As $t \to 0$, the functions $g_t$ ``concentrate'' on $\om$.}
\vspace{10pt}

    We begin by observing that $\di_{\oc} g_t^2e^\frac{-V}{t}$ and $\di_{\oc} \lVert \nabla g_t \rVert^2e^\frac{-V}{t}$ of order $t^\frac{1}{\alpha}$, and that $$\di_{\om} g_t^2 \to 1.$$
Indeed, let $h_t=g_t-\frac{1}{|\om|}\di_\om g_t$, applying Lemma \ref{magicnd}, then recalling that we have a continuous trace embedding $H^1(\om) \to L^2(\po)$, it holds that, for $t$ sufficiently small,

\allowdisplaybreaks
\begin{align*}
\begin{split}
        C_{\mathrm{P}}(t)&\leq \dfrac{\di_\om h_t(x)^2\,\dd x+\di_{\oc}h_t(x)^2e^\frac{-V(x)}{t}\,\dd x}{\di_{\om} \lVert\nabla h_t(x)\rVert^2\,\dd x+\di_{\oc}\lVert\nabla h_t(x)\rVert^2 e^\frac{-V(x)}{t}\,\dd x} \\
        &\leq \dfrac{\di_\om h_t(x)^2\,\dd x+2(Z_t-|\om|)\di_{\po}h_t(\theta)^2\dd\tilde{\sigma}(\theta)+o(t^\frac{1}{\alpha})\times \di_{\oc}\lVert\nabla h_t(x)\rVert^2 e^\frac{-V(x)}{t}\,\dd x +o(t^\frac{1}{\alpha})}{\di_{\om} \lVert\nabla h_t(x)\rVert^2\,\dd x+\di_{\oc}\lVert\nabla h_t(x)\rVert^2 e^\frac{-V(x)}{t}\,\dd x}\\
        &\leq \dfrac{\di_\om h_t(x)^2\,\dd x}{\di_{\om} \lVert\nabla h_t(x)\rVert^2\,\dd x+\di_{\oc}\lVert\nabla h_t(x)\rVert^2 e^\frac{-V(x)}{t}\,\dd x}+O(t^\frac{1}{\alpha})\\
        &\leq \dfrac{\di_\om h_t(x)^2\,\dd x}{\di_{\om} \lVert\nabla h_t(x)\rVert^2\,\dd x}+O(t^\frac{1}{\alpha}) \leq 1+O(t^\frac{1}{\alpha}).
    \end{split}
\end{align*} 
Since $C_{\mathrm{P}}(t) \geq 1+O(t^\frac{1}{\alpha})$ (see the lower bound given in Step 3), it yields $$C_{\mathrm{P}}(t)=1+O(t^\frac{1}{\alpha}) ,\quad\di_\om h_t^2(x)\,\dd x=1+O(t^\frac{1}{\alpha})$$ and also $$\dfrac{\di_\om h_t(x)^2\,\dd x}{\di_{\om} \lVert\nabla h_t(x)\rVert^2\,\dd x+\di_{\oc}\lVert\nabla h_t(x)\rVert^2 e^\frac{-V(x)}{t}\,\dd x}=1+O(t^\frac{1}{\alpha})$$
which implies $\di_{\oc}\lVert\nabla g_t(x)\rVert^2 e^\frac{-V(x)}{t}\,\dd x=\di_{\oc}\lVert\nabla h_t(x)\rVert^2 e^\frac{-V(x)}{t}\,\dd x=O(t^\frac{1}{\alpha})$.

Besides,
\begin{equation}
    \begin{split}
        \bigg(\di_\om g_t\bigg)^2&= \bigg(\di_{\oc} g_te^\frac{-V}{t}\bigg)^2\\
        &\leq \bigg(\di_{\oc} g_t^2e^\frac{-V}{t}\bigg)\bigg(\di_{\oc} 1^2e^\frac{-V}{t}\bigg) \\
        &\leq \Big(2(Z_t-\lvert \om \rvert )\di_{\po}g_t(\theta)^2\,\dd\tilde{\sigma}(\theta)+o(t^\frac{1}{\alpha})\Big) \times (Z_t-\lvert \om \rvert) \text{~~~(by Lemma \ref{magicnd}})\\
        & \leq O(t^\frac{2}{\alpha})
    \end{split}
    \nonumber
\end{equation}
where we used the continuity of the trace embedding $H^1(\om) \to L^2(\po)$ and the fact that $Z_t=\lvert \om \rvert +O(t^\frac{1}{\alpha})$.

\vspace{10pt}

\noindent\textbf{Step 2.} \textit{Upper bound on $C_{\mathrm{P}}(t)$.}
\vspace{10pt}

Since $\di_\om \lVert \nabla g_t \rVert^2\,\dd \mu_t=1$, we know by Lemma \ref{magicnd} that  $$\di_{\oc}g_t^2e^\frac{-V}{t}=(Z_t-\lvert \om \rvert )\di_{\po}g_t(\theta)^2\dd\tilde{\sigma}(\theta) + o(t^\frac{1}{\alpha})$$
and recalling Lemma \ref{pasencore},
$$\di_{\oc} \lVert \nabla g_t\rVert^2e^\frac{-V}{t}=(Z_t-\lvert \om \rvert )\di_{\po}\lVert \nabla g_t(\theta)\rVert ^2\dd\tilde{\sigma}(\theta) + o(t^\frac{1}{\alpha}).$$
Therefore,    
\begin{equation}
    \begin{split}
C_{\mathrm{P}}(t)&=\dfrac{\di_\om g_t(x)^2\,\dd x+(Z_t-\lvert \om \rvert )\di_{\po}g_t(\theta)^2\dd\tilde{\sigma}(\theta)}{\di_{\om} \lVert\nabla g_t(x)\rVert^2\,\dd x+(Z_t-\lvert \om \rvert )\di_{\po}\lVert \nabla g_t(\theta)\rVert ^2\dd\tilde{\sigma}(\theta)}+o(t^\frac{1}{\alpha})\\
&= \di_\om g_t(x)^2\,\dd x+(Z_t-\lvert \om \rvert)\Bigg(\di_{\po}g_t(\theta)^2\dd\tilde{\sigma}(\theta)-\di_{\po}\lVert \nabla g_t(\theta)\rVert ^2\dd\tilde{\sigma}(\theta)\Bigg)+o(t^\frac{1}{\alpha})\\
&\leq 1+(Z_t-\lvert \om \rvert)\Bigg(\di_{\po}f(\theta)^2\dd\tilde{\sigma}(\theta)-\di_{\po}\lVert \nabla f(\theta)\rVert ^2\dd\tilde{\sigma}(\theta)\Bigg)+o(t^\frac{1}{\alpha})\\
    \end{split}
    \nonumber
\end{equation}
where, in the last inequality, we have applied the Poincaré inequality satisfied by $\mu_0 = \dfrac{\,\dd x_{|\om}}{\lvert \om \rvert}$ and used that $\Big(\di_\om g_t\Big)^2=o(t^\frac{1}{\alpha})$. Moreover, by the simplicity\footnote{We assumed in the beginning of the proof that the smallest nonzero eigenvalue is simple, but this assumption is not mandatory : in general, the corresponding eigenspace is finite-dimensional and the same argument holds by considering the accumulation points of $(g_t)_{t > 0}$ as $t\to 0^+$.} of the smallest Neumann eigenvalue, we have (up to a sign change for every $g_t$) ${g_t}_{|\om} \xrightarrow{H^1(\om)}f$ (as $\di_\om g_t(x)\,\dd x \to 0$),
and by continuity of the trace embedding $H^1(\om) \to L^2(\partial\om)$, we also have ${g_t}_{|\partial \om} \xrightarrow[]{L^2(\partial\om)} f_{|\partial\om}$, which further implies, since $(g_t)_{t> 0}$ is uniformly bounded in  $H^2(\om)$, that $$\di_{\po}\lVert \nabla f(\theta)\rVert ^2\dd\tilde{\sigma}(\theta) \leq \underset{t \to 0}{\liminf} \di_{\po}\lVert \nabla g_t(\theta)\rVert ^2\dd\tilde{\sigma}(\theta).$$

\noindent\textbf{Step 3.} \textit{Lower bound on $C_{\mathrm{P}}(t)$.}
\vspace{10pt}

The function $f$ being a Neumann eigenfunction on a smooth domain $\om$, it holds that $f \in C^\infty(\bar{\om})$. We can therefore extend it to a function defined on all of $\mathbb{R}^n$, such that $\lVert f \rVert_{C^2(\mathbb{R}^n)} < \infty$. It follows from this condition that
 $$\di_{\oc}f^2e^\frac{-V}{t}=(Z_t-\lvert \om \rvert )\di_{\po}f(\theta)^2\dd\tilde{\sigma}(\theta) + o(t^\frac{1}{\alpha})$$
 and (recalling Lemma \ref{pasencore}) that 
 $$\di_{\oc} \lVert \nabla f\rVert^2e^\frac{-V}{t}=(Z_t-\lvert \om \rvert )\di_{\po}\lVert \nabla f(\theta)\rVert ^2\dd\tilde{\sigma}(\theta) + o(t^\frac{1}{\alpha}).$$
 Note that $\bigg\lvert \di_{\mathbb{R}^n} fe^\frac{-V}{t}\bigg\rvert =\bigg\lvert\di_{\oc} fe^\frac{-V}{t}\bigg\rvert \leq \lVert f \rVert_\infty \times O(t^\frac{1}{\alpha})$, therefore,
\begin{equation}
    \begin{split}
C_{\mathrm{P}}(t)&\geq\dfrac{\di_\om f(x)^2\,\dd x+(Z_t-\lvert \om \rvert )\di_{\po}f(\theta)^2\dd\tilde{\sigma}(\theta)}{\di_{\om} \lVert\nabla f(x)\rVert^2\,\dd x+(Z_t-\lvert \om \rvert )\di_{\po}\lVert \nabla f(\theta)\rVert ^2\dd\tilde{\sigma}(\theta)}+o(t^\frac{1}{\alpha})\\
&= 1+(Z_t-\lvert \om \rvert)\Bigg(\di_{\po}f(\theta)^2\dd\tilde{\sigma}(\theta)-\di_{\po}\lVert \nabla f(\theta)\rVert ^2\dd\tilde{\sigma}(\theta)\Bigg)+o(t^\frac{1}{\alpha}).\\
    \end{split}
    \nonumber
\end{equation}
This completes the proof.    
\end{proof}

\section{Low-temperature asymptotics of log-Sobolev constants in dimension $1$}\label{LSIasympt}

This section is devoted to the proof of Theorem \ref{lsi1d}. The proof consists in analyzing functions that nearly saturate the log-Sobolev inequality satisfied by $\mu_t$, in order to derive precise information about the corresponding optimal constant. In our setting, near-extremizers fall within the linearization regime that connects the log-Sobolev inequality to the spectral gap. In other words, they are close to constant functions, and the structure of their fluctuations dictates the asymptotic behavior $C_{LS}(\mu_t)$.

Before proving Theorem \ref{lsi1d}, we recall a useful result that we use to derive log-Sobolev inequalities. Indeed, one way to show that a probability measure satisfies a log-Sobolev inequality is to tighten a defective-LSI with a Poincaré inequality. This is known as the Rothaus lemma \cite{ROTHAUS1985296}. More precisely,

\begin{lemma}\label{Rothaus}
    Let $\mu$ be a probability measure on $\mathbb{R}^n$ satisfying a defective LSI, i.e.
    $$\mathrm{Ent}_\mu(f^2) \leq 2A \di \lVert \nabla f \rVert^2\,\dd \mu+B \di f^2\,\dd \mu$$
    for all regular functions $f$, where $A,B$ are nonnegative constants.
    Then,
    $$C_{\mathrm{LS}}(\mu) \leq A + \frac{B}{2}C_{\mathrm{P}}(\mu).$$
\end{lemma}

\vspace{10pt}

\newcommand{\fii}{\varphi}
\newcommand{\md}{m_r^{(t)}}
\newcommand{\mg}{m_l^{(t)}}
\newcommand{\mdn}{{\Tilde{m}}_r^{(t)}}
\newcommand{\mgn}{{\Tilde{m}}_l^{(t)}}
\newcommand{\at}{r_t}
\newcommand{\bt}{l_t}
\newcommand{\ct}{C_{\mathrm{LS}}(\mu_t)}
\newcommand{\st}{\sqrt{t}}
\newcommand{\fn}{\Tilde{f_t}}
\newcommand{\gn}{\Tilde{g_t}}
\newcommand{\tn}{{t_n}}

\begin{proofof}{Theorem \ref{lsi1d}}
Let $V$ be a potential satisfying the assumptions of Theorem \ref{lsi1d} and $[a,b] = \mathrm{argmin}(V)$.

In what follows, all implicit constants in big-$O$ and small-$o$ notation depend only on $V$.
\medskip

\noindent\textbf{Step 1.} \textit{General properties of $\mu_t$.}

Let us introduce the following notations
\begin{itemize}[itemsep=2pt, topsep=2pt]
\item $\fii : x \mapsto x\log(x)$
\item $I=[a,b],\quad R=[b,+\infty),\quad L= (-\infty, a]$
\item $r_t= \mu_t(R)=\mu_t\big([b,+\infty)\big),\quad l_t=\mu_t(L)=\mu_t\big((-\infty,a]\big)$
\item$\md(h)= \dfrac{\di_R h^2\,\dd \mu_t}{\mu_t(R)}=\dfrac{\di_{[b,+\infty)} h^2\,\dd \mu_t}{r_t},\quad\mg(h)= \dfrac{\di_L h^2\,\dd \mu_t}{\mu_t(L)}=\dfrac{\di_{[b,+\infty)} h^2\,\dd \mu_t}{l_t}$
\end{itemize}
for all functions $h$ for which the above expressions are meaningful. Note that the assumptions of Theorem \ref{lsi1d} imply that $l_t \underset{t\to0}{\sim} \alpha\sqrt{t}$ and that $r_t \underset{t\to 0}{\sim} \beta\sqrt{t}$ where $\alpha = \dfrac{1}{b-a}\sqrt{\dfrac{\pi}{2V''(a^-)}}$ and $\beta=\dfrac{1}{b-a}\sqrt{\dfrac{\pi}{2V''(b^+)}}$.

\medskip
Let $h$ be any compactly supported smooth function. Observe that by Theorem \ref{CS}, we have, at least for small values of $t$,
\begin{align}\label{Ic}
\begin{split}
\fii(\md(h)) \leq \dfrac{1}{\at}\di_R \fii(h^2)\,\dd \mu_t \leq \fii(\md(h))+Ct \cdot \dfrac{1}{\at} \di_R \lvert  h' \rvert^2\,\dd \mu_t\\
\fii(\mg(h)) \leq \dfrac{1}{\bt}\di_L \fii(h^2)\,\dd \mu_t \leq \fii(\mg(h))+Ct \cdot \dfrac{1}{\bt} \di_L \lvert  h' \rvert^2\,\dd \mu_t
\end{split}
\end{align}
for $t$ sufficiently small, where $C$ is a positive constant depending only on $V$.

Moreover, by Remark \ref{generalizationmeancontrol}, for $t$ sufficiently small and $0<\varepsilon<1$ it holds that
\begin{align}\label{epsilon}
\begin{split}
\lvert \mg(h) - h(\small a)^2\rvert &\leq \varepsilon h(\small a)^2+\dfrac{2}{\varepsilon}C_L t \Big(\frac{1}{\bt}\di_L  \lvert h' \rvert^2\,\dd \mu_t\Big)\\
    \lvert \md(h) - h(\small b)^2\rvert &\leq \varepsilon h(\small b)^2+\dfrac{2}{\varepsilon}C_R t \Big(\frac{1}{\at}\di_R  \lvert h' \rvert^2\,\dd \mu_t\Big)
\end{split}
\end{align}
where $C_R,C_L$ are positive constants. Taking $\varepsilon=t^{\frac{1}{4}}$, recalling that $r_t \sim\beta\st$ and $\bt \sim \alpha\st$, and observing that $\lVert h_{|I} \rVert_\infty^2=O\big(\di_I h^2\,\dd \mu_t+\di_I \lvert h' \rvert^2\,\dd \mu_t\big)$ yields
\begin{align}\label{meancontrol}
    \begin{split}
        \md(h)=h(b)^2+O\Bigg(t^{\frac{1}{4}}\bigg(\di h^2\,\dd \mu_t+\di \lvert h' \rvert^2\,\dd \mu_t\bigg)\Bigg)\\
        \mg(h)=h(a)^2+O\Bigg(t^{\frac{1}{4}}\bigg(\di h^2\,\dd \mu_t+\di \lvert h' \rvert^2\,\dd \mu_t\bigg)\Bigg)
    \end{split}
\end{align}

Finally, Lemma \ref{Rothaus} implies $C_{\mathrm{LS}}(\mu_t)<\infty$ for $t>0$ sufficiently small. Indeed, when $h$ is a compactly supported smooth function such that $\di h^2\,\dd \mu =1$, using the estimates \eqref{Ic} above and Proposition \ref{segment} yields
\allowdisplaybreaks
\begin{align*}
\begin{split}
    \di \varphi(h^2)\,\dd \mu_t &= \di_R \varphi(h^2)\,\dd \mu_t+\di_L \varphi(h^2)\,\dd \mu_t+\di_a^b \varphi(h^2)\,\dd \mu_t \\ &\leq r_t\varphi(\md(h))+Ct\di_R \lvert  h' \rvert^2\,\dd \mu_t\\&+\bt\varphi(\mg(h))+Ct\di_L \lvert  h' \rvert^2\,\dd \mu_t\\&+\mu_t([a,b])\varphi\Big(\frac{1}{\mu_t([a,b])}\di_I h^2\,\dd \mu_t\Big)+\frac{(b-a)^2}{\pi^2}\di_a^b  \lvert  h' \rvert^2.
\end{split}
\end{align*}
Observing that $\md(h) \leq \dfrac{1}{\at},~\mg(h) \leq \dfrac{1}{\bt}$ and $\frac{1}{\mu_t([a,b])}\di_I h^2\,\dd \mu_t\leq \dfrac{1}{\mu_t([a,b])} $, we get
\begin{equation}
    \di \varphi(h^2)\,\dd \mu_t \leq A_t\di \lvert h' \rvert^2\,\dd \mu_t +B_t
\end{equation}
where $A_t = \max(Ct,\frac{(b-a)^2}{\pi})$ and $B_t=\lvert r_t\fii(\frac{1}{r_t})\rvert+\lvert l_t\fii(\frac{1}{l_t})\rvert+\lvert\frac{1}{\mu_t([a,b])}\fii(\mu_t([a,b])\rvert$ are independent of $h$, which allows us to apply the Rothaus lemma.

\medskip

\noindent\textbf{Step 2.} \textit{Convergence of near-optimal functions.}

For $t>0$, let $f_t$ be a nonnegative compactly supported smooth function such that
    $$2C_{\mathrm{LS}}(\mu_t)+o(\st)=\dfrac{\di \fii(f_t^2)\,\dd \mu_t-\fii\Big(\di f_t^2\,\dd \mu_t\Big)}{\di \lvert f_t' \rvert^2\,\dd \mu_t}.$$

    By homogeneity of the logarithmic Sobolev inequality, we may assume without loss of generality that $\di f_t^2\,\dd \mu_t=1$.

Assume, for the sake of contradiction, that $\di \lvert f_t' \rvert^2\,\dd \mu_t$ is not bounded as $t \to 0$. Therefore, up to subsequence, $\di \lvert f_t' \rvert^2\,\dd \mu_t \xrightarrow[]{t\to 0}+\infty$. Let $g_t = \dfrac{f_t}{\sqrt{\int \lvert f_t' \rvert^2}}$ so that $\di g_t^2\,\dd \mu_t \to 0$ and $\di \lvert g_t' \rvert^2\,\dd \mu_t =1$. 
Since $\di_I g_t^2(x)\,\dd x \to 0$ and $\di_I\lvert g_t'(x)\rvert^2\,\dd x$ is bounded, it holds that ${g_t}_{|I}$ is uniformly bounded in $L^{\infty}(I)$ as $t \to 0$.
Hence,$\di_I\fii(g_t^2)\,\dd \mu_t \to 0$ as $t \to 0$. Indeed,
\begin{align*}
\bigg\lvert\di_I\fii(g_t^2)\,\dd \mu_t\bigg\rvert &\leq \di_I \underset{I}{\sup} \lvert g_t\log(g_t)\rvert\times \lvert g_t(x)\rvert \,\dd x \\
&\leq \underset{I}{\sup} \lvert g_t\log(g_t)\rvert \di_I \lvert g_t(x) \rvert \,\dd x \\
&\leq \big(1+\fii(\lVert {g_t}_{|I} \rVert_{L^\infty})\big) \sqrt{(b-a)\di_I g_t^2(x)\,\dd x}~ \longrightarrow 0 \text{ as $t \to 0$.}
\end{align*}
Therefore, applying \eqref{Ic} then \eqref{meancontrol},
\allowdisplaybreaks
\begin{align}\label{conv0}
\nonumber
    2\ct &\leq \di \fii(g_t^2)\,\dd \mu_t -\fii\bigg(\di g_t^2\,\dd \mu_t\bigg) +o(\sqrt{t})\\\nonumber
    &= \di_I \fii(g_t^2) \,\dd \mu_t + \di_R \fii(g_t^2) \,\dd \mu_t + \di_L \fii(g_t^2) \,\dd \mu_t -\fii\bigg(\di g_t^2\,\dd \mu_t\bigg)+o(1)\\\nonumber
    &=  \di_R \fii(g_t^2) \,\dd \mu_t + \di_L \fii(g_t^2)\,\dd \mu_t +o(1)\\\nonumber
    & \leq \at \fii(\md(g_t))+\bt\fii(\mg(g_t)) + Ct \di_{I^c} \lvert g_t' \rvert ^2\,\dd \mu_t +o(1)\\\nonumber
    &= \at \fii\Bigg[g_t(b)^2+O\Big(t^{\frac{1}{4}}\big(\di g_t^2\,\dd \mu_t+\di \lvert g_t' \rvert^2\,\dd \mu_t\big)\Big)\Bigg)\Bigg]\\\nonumber
    &+ \bt \fii\Bigg[g_t(a)^2+O\Big(t^{\frac{1}{4}}\big(\di g_t^2\,\dd \mu_t+\di \lvert g_t' \rvert^2\,\dd \mu_t\big)\Big)\Bigg)\Bigg] +o(1)\\\nonumber
    &=\at \fii(g_t(b)^2+o(1)) + \bt\fii(g_t(a)^2+o(1)) +o(1)\\
    &=O(\st)+o(1) \text{ since $\lVert {g_t}_{|I} \rVert_\infty^2=O\big(\di_I g_t^2\,\dd \mu_t+\di_I \lvert g_t' \rvert^2\,\dd \mu_t\big)=O(1)$}.
    \end{align}
However, from Section \ref{prelim}, and recalling Theorem \ref{1d}, for $t$ sufficiently small, we have
\begin{equation}\label{borneinf}
    \ct \geq C_{\mathrm{P}}(\mu_t)\geq \dfrac{(b-a)^2}{\pi^2}
\end{equation}
which is in contradiction with \eqref{conv0}.

\medskip
Hence, $\di \lVert f_t'\rVert^2\,\dd \mu_t$ is bounded as $t \to 0$. Recall that $\di f_t^2\,\dd \mu_t=1$ so that $\lVert{f_t}_{|I}\rVert_\infty$ is bounded as $t \to 0$.
For brevity, in the sequel, we shall write $\md$ for $\md(f_t)$ and $\mg$ for $\mg(f_t)$. Note that by \eqref{meancontrol} applied to $f_t$, $\md$ and $\mg$ are bounded as $t \to 0$.
More precisely,
$\md = f_t(b)^2+o(1)$ and $\mg=f_t(a)^2+o(1)$.

\noindent As a consequence, since $r_t+l_t=O(\st)$, \begin{equation}\label{cvL2I}
\di_I f_t^2\,\dd \mu_t = \di f_t^2\,\dd \mu_t-\at\md-\bt\mg=1+O(\st).
\end{equation}
Now, applying \eqref{Ic} and recalling that $\at+\bt=O(\st)$,
\begin{align}\label{H1conv}
\nonumber
2\ct&=  \dfrac{\di_I\fii(f_t^2)\,\dd \mu_t+\di_R \fii(f_t^2) \,\dd \mu_t + \di_L \fii(f_t^2) \,\dd \mu_t }{\di \lvert  f_t' \rvert^2\,\dd \mu_t}+o(\st)\\ 
\nonumber
&\leq \dfrac{\di_I\fii(f_t^2)\,\dd \mu_t+\at\fii(\md)+\bt\fii(\mg)+Ct\di_{I^c} \lvert f_t' \rvert ^2\,\dd \mu_t}{\di \lvert  f_t' \rvert^2\,\dd \mu_t}+o(\st)\\
\nonumber
& = \dfrac{\di_I\fii(f_t^2)\,\dd \mu_t+O(\sqrt{t})+O(t)}{\di \lvert  f_t' \rvert^2\,\dd \mu_t}+o(\st)\\
&=\dfrac{\di_I\fii(f_t^2)\,\dd \mu_t+o(1)}{\di \lvert  f_t' \rvert^2\,\dd \mu_t}+o(1) \leq \dfrac{\di_I\fii(f_t^2)\,\dd \mu_t+o(1)}{\di_I \lvert  f_t' \rvert^2\,\dd \mu_t}+o(1).
\end{align}

Let $f : I \to \mathbb{R}$ be any weak-$H^1(I)$, strong-$L^\infty(I)$ accumulation point of ${f_t}_{|I}$ as $t \to 0$. If $f$ is not a constant, then
$$0<\di_I \lvert f'(x) \rvert^2 \,\dd x\leq \liminf \di_I \lvert f_t'(x)\rvert^2 \,\dd x = (b-a)\liminf \di_I \lvert f_t' \rvert^2\,\dd \mu_t$$ (where the $\liminf$ is taken along any subsequence converging to $f$). Moreover,
\begin{align*}
&\di_I\varphi(f^2(x))\,\dd x = \lim \di_I\varphi(f_t^2(x))\,\dd x = (b-a)\lim \di_I\fii(f_t^2)\,\dd \mu_t,\\
&\dfrac{1}{b-a}\di_I f^2(x)\,\dd x = \lim \di_If_t^2\,\dd \mu_t\leq 1
\end{align*}
from which it follows, as $\fii$ takes nonpositive values on $[0,1]$, that $f$ attains equality in the logarithmic Sobolev inequality satisfied by $\mu_0$ as
\begin{align*}
    2C_{\mathrm{LS}}(\mu_0)\leq \liminf 2\ct &\leq \dfrac{\di_I\fii(f^2)\,\dd \mu_0}{\di_I \lvert  f' \rvert^2\,\dd \mu_0}\leq \dfrac{\di\fii(f^2)\,\dd \mu_0-\fii\bigg(\di f^2\,\dd \mu_0\bigg)}{\di_I \lvert  f' \rvert^2\,\dd \mu_0}.
\end{align*}
which contradicts the nonexistence of optimal functions for the logarithmic Sobolev inequality on a segment (see Remark \ref{weiss}).

\noindent Therefore, any weak-$H^1(I)$, strong-$L^\infty(I)$ accumulation point of ${f_t}_{|I}$ as $t \to 0$ is constant.
Let $c$ be such a constant, then, by \eqref{cvL2I}, 
$$c^2 = \lim \di_If_t^2\,\dd \mu_t=1.$$
By compactness of the embedding $H^1(I) \xrightarrow[]{} C(I)$, since $f_t \geq 0$ for all $t>0$,
\begin{align*}
{f_t}_{|I} \xrightarrow[]{L^\infty} 1.
\end{align*}
Using then \eqref{borneinf} and \eqref{H1conv},
\begin{align*}
    \dfrac{2(b-a)^2}{\pi^2}+o(1)\leq 2\ct &\leq \dfrac{\di_I\fii(f_t^2)\,\dd \mu_t+o(1)}{\di \lvert  f_t' \rvert^2\,\dd \mu_t}+o(1)\leq \dfrac{o(1)}{\di \lvert  f_t' \rvert^2\,\dd \mu_t}+o(1)
\end{align*}
which leads to
$$\di \lvert  f_t' \rvert^2\,\dd \mu_t \xrightarrow[]{}0.$$

Thus, we have proved the following facts:
\begin{enumerate}[label=(\roman*), itemsep=2pt, topsep=2pt]
\item ${f_t}_{|I} \xrightarrow[]{L^\infty}1$
\item $\md \to 1$ (as $\md=f_t(b)^2+o(1)$)
\item $\mg \to 1$
\item $\di \lvert  f_t' \rvert^2\,\dd \mu_t \xrightarrow[]{}0$.
\end{enumerate}

\vspace{15pt}
\noindent\textbf{Step 3.} \textit{Reduction to a problem on $I$}

Now, we shall prove that we may assume that $f_t$ is constant on $R$ and on $L$, exactly like in the proof of Theorem \ref{1d}.

Up to rescaling $f_t$ by a factor of $1+o(1)$, by homogeneity of the logarithmic Sobolev inequality, we may assume without loss of generality $$\dfrac{1}{b-a}\di_If_t(x)\,\dd x=1.$$
Note that facts (i)–(iv) remain valid, however, although $\di f_t^2\,\dd \mu_t$ is no longer equal to $1$, we still have $\di f_t^2\,\dd \mu_t=1+o(1)$.
Let $\fn$ the function defined by
\begin{equation*}
    \fn(x)=f_t(\mathrm{proj_{[a,b]}}(x))=
    \begin{cases}
        f_t(x) &\text{ if $x \in I$}\\
        f_t(a) &\text{ if $x < a$}\\
        f_t(b) &\text{ if $x > b$}
    \end{cases}
\end{equation*}
We shall prove that
$$2C_{\mathrm{LS}}(\mu_t)=\dfrac{\di \fii(\fn^2)\,\dd \mu_t-\fii\Big(\di \fn^2\,\dd \mu_t\Big)}{\di \lvert \fn' \rvert^2\,\dd \mu_t}+o(\st).$$

Define $g_t$ as $f_t-1$. Let us first observe that by Theorem \ref{CS},
\begin{equation}\label{poincarépl}
    \begin{split}
        \di_R (g_t-g_t(b))^2\,\dd \mu_t \leq c_Rt\di_R \lvert g_t' \rvert^2\,\dd \mu_t
    \end{split}
\end{equation}
for $t$ sufficiently small, where $c_R$ is a positive constant. Therefore,

\begin{equation}\label{order1}
    \begin{split}
        \bigg\lvert\di_R g_t\,\dd \mu_t-\at g_t(b)\bigg\rvert &\leq \di_R \lvert g_t- g_t(b)\rvert\,\dd \mu_t\\
        & \leq \sqrt{\at\di_R (g_t-g_t(b))^2\,\dd \mu_t}~\text{ (by Cauchy-Schwarz)}\\
        & \leq \sqrt{\at c_Rt\di_R \lvert g_t'\rvert ^2\,\dd \mu_t} ~\text{ (by \eqref{poincarépl})}\\
        &= O(t^\frac{3}{4})\sqrt{\di_R \lvert g_t'\rvert ^2\,\dd \mu_t}.
    \end{split}
\end{equation}
Moreover,
\begin{equation}
    \begin{split}
        \at\md &= \di_R f_t^2\,\dd \mu_t=\di_R(1+g_t(b)+g_t-g_t(b))^2\,\dd \mu_t\\
        &=\at(1+g_t(b))^2+\di_R (g_t-g_t(b))^2\,\dd \mu_t  + 2 \di_R g_t -g_t(b)\,\dd \mu_t +2g_t(b)\di_R g_t -g_t(b)\,\dd \mu_t.
    \end{split}
    \nonumber
\end{equation}
Hence
\allowdisplaybreaks
\begin{align}\label{produit}\nonumber
       &\big \lvert\md-f_t(b)^2-\frac{2}{\at}\di_R g_t -g_t(b)\,\dd \mu_t\big\rvert\leq\frac{1}{r_t}\di_R (g_t-g_t(b))^2\,\dd \mu_t +\big\lvert\frac{2}{r_t}g_t(b)\di_R g_t -g_t(b)\,\dd \mu_t\big\rvert\\ \nonumber
       &\leq c_R\frac{t}{\at}\di_R\lvert g_t' \rvert^2\,\dd \mu_t+\frac{1}{r_t}\bigg(t^{\frac{3}{4}}g_t(b)^2+t^{-\frac{3}{4}}\Big(\di_R g_t -g_t(b)\,\dd \mu_t\Big)^2\bigg)\\ \nonumber
       &\leq O(\sqrt{t})\di_R\lvert g_t' \rvert^2\,\dd \mu_t+O(t^{\frac{1}{4}})g_t(b)^2+O(\frac{1}{r_t}t^{-\frac{3}{4}}.r_tc_Rt)\di_R \lvert g_t'\rvert ^2\,\dd \mu_t\\ \nonumber
       &\leq O(t^{\frac{1}{4}})g_t(b)^2+o(1)\di_R \lvert g_t'\rvert ^2\,\dd \mu_t\\
       &\leq O(t^{\frac{1}{4}})\di_I\lvert g_t'\rvert^2\,\dd \mu_t + o(1)\di_R \lvert g_t'\rvert ^2\,\dd \mu_t
\end{align}
where, in the last inequality, we used that $\lVert {g_t}_{|I}^2 \rVert_\infty=O\Big(\di_I \lvert g_t' \rvert^2\,\dd \mu_t\Big)$ since $\di_Ig_t = 0$.

Let $z_t=\md-f_t(b)^2$ so that $z_t \to 0$, and more precisely, we can rewrite the inequality above as
\begin{align}\label{maj1}
    \begin{split}
        \Big\rvert z_t - \frac{2}{\at}\di_R g_t -g_t(b)\,\dd \mu_t \Big\rvert
        \leq O(t^{\frac{1}{4}})\di_I\lvert g_t'\rvert^2\,\dd \mu_t + o(1)\di_R \lvert g_t'\rvert ^2\,\dd \mu_t
    \end{split}
\end{align}
yielding
\begin{align}\label{maj2}
    \begin{split}
        z_t^2&= O(\frac{1}{\at^2})\Big(\di_R g_t -g_t(b)\,\dd \mu_t\Big)^2 + O(t^{\frac{1}{2}})\Big(\di_I\lvert g_t'\rvert^2\,\dd \mu_t\Big)^2 + o(1)\Big(\di_R \lvert g_t'\rvert ^2\,\dd \mu_t\Big)^2\\
        &= o(1)\di_R \lvert g_t'\rvert ^2\,\dd \mu_t+o(\st)\di_I \lvert g_t'\rvert ^2\,\dd \mu_t
    \end{split}
\end{align}
so that
\begin{align}\label{logmoyennedroite}\nonumber
    &\at\fii(\md)=r_t\fii(f_t(b)^2+z_t)=\at\Big(\fii(f_t(b)^2)+\fii'(f_t(b)^2)z_t +O(1)z_t^2\Big)\\\nonumber
        &=r_t\fii(f_t(b)^2)+r_t\big(1+O(g_t(b)\big)z_t+o(1)\di_R \lvert g_t'\rvert ^2\,\dd \mu_t+o(\st)\di_I \lvert g_t'\rvert ^2\,\dd \mu_t~\text{ (by \eqref{maj2})}\\ \nonumber
        &=\at\fii(f_t(b)^2)\\ \nonumber
        &+\big(1+O(g_t(b)\big)\Big(2\di_R g_t -g_t(b)\,\dd \mu_t +r_t O(t^{\frac{1}{4}})\di_I\lvert g_t'\rvert^2\,\dd \mu_t + r_t o(1)\di_R \lvert g_t'\rvert ^2\,\dd \mu_t\Big)\\\nonumber
        &+o(1)\di_R \lvert g_t'\rvert ^2\,\dd \mu_t+o(\st)\di_I \lvert g_t'\rvert ^2\,\dd \mu_t ~\text{ (by \eqref{maj1})}\\\nonumber
        &=r_t\fii(f_t(b)^2)+2(1+O(g_t(b)))\Big(\di_R g_t -g_t(b)\,\dd \mu_t\Big)
        +o(1)\di_R \lvert g_t'\rvert ^2\,\dd \mu_t+o(\st)\di_I \lvert g_t'\rvert ^2\,\dd \mu_t\\
        &=r_t\fii(f_t(b)^2)+2\di_R g_t -g_t(b)\,\dd \mu_t +o(1)\di_R \lvert g_t'\rvert ^2\,\dd \mu_t+o(\st)\di_I \lvert g_t'\rvert ^2\,\dd \mu_t
\end{align}
where we used $g_t(b)\di_R g_t -g_t(b)\,\dd \mu_t=o(1)\di_R \lvert g_t'\rvert ^2\,\dd \mu_t+o(\st)\di_I \lvert g_t'\rvert ^2\,\dd \mu_t$ following the argument in \eqref{produit}.

Similarly, we have
\begin{equation}\label{logmoyennegauche}
    \bt\fii(\mg)=\bt\fii(f_t(b)^2)+2\di_L g_t -g_t(a)\,\dd \mu_t +o(1)\di_L \lvert g_t'\rvert ^2\,\dd \mu_t+o(\st)\di_I \lvert g_t'\rvert ^2\,\dd \mu_t
\end{equation}

We now investigate the term $\fii\big(\di f_t^2\,\dd \mu_t\big)$. Indeed,

\begin{equation*}
    \begin{split}
        \fii\big(\di f_t^2 \mu_t\big)&=\fii\big(\di(1+g_t)^2\,\dd \mu_t\big)\\
        &=\fii\big(1+2\di g_t\,\dd \mu_t +\di g_t^2\,\dd \mu_t\big)\\
        &=\fii\bigg(1+2\at g_t(b)+\at g_t(b)^2+2\bt g_t(a) +\bt g_t(a)^2 \\
        &+\di_R 2(g_t-g_t(b))+g_t^2-g_t(b)^2\,\dd \mu_t+\di_L2(g_t-g_t(a))+g_t^2-g_t(a)^2\,\dd \mu_t\bigg)
    \end{split}
\end{equation*}
But by \eqref{epsilon} for $\varepsilon=t^\frac{1}{4}$,
\begin{equation*}
    \begin{split}
        \bigg\lvert \di_R g_t^2-g_t(b)^2\,\dd \mu_t\bigg\rvert &=r_t\lvert \md(g)-g_t(b)^2\rvert\\& \leq O(t^\frac{3}{4})g_t(b)^2+o(1)\di_R \lvert g_t'\rvert^2\,\dd \mu_t\\&=o(\st)\di_I \lvert g_t'\rvert ^2\,\dd \mu_t+o(1)\di_R \lvert g_t'\rvert ^2\,\dd \mu_t
    \end{split}
\end{equation*}
and we similarly obtain
$$\bigg\lvert \di_L g_t^2-g_t(a)^2\,\dd \mu_t\bigg\rvert =o(\st)\di_I \lvert g_t'\rvert ^2\,\dd \mu_t+o(1)\di_L \lvert g_t'\rvert ^2\,\dd \mu_t.$$
Since $\di f_t^2- \fn^2\,\dd \mu_t = \di_R 2(g_t-g_t(b))+g_t^2-g_t(b)^2\,\dd \mu_t +\di_L 2(g_t-g_t(a))+g_t^2-g_t(a)^2\,\dd \mu_t $,
\begin{equation*}
    \begin{split}
        \fii\Big(\di f_t^2 \mu_t\Big)=\fii\bigg(\di \fn^2\,\dd \mu_t+\di_R 2(g_t-g_t(b))\,\dd \mu_t+\di_L 2(g_t-g_t(a))\,\dd \mu_t+E_t\bigg)
\end{split}
\end{equation*}
where 
\begin{equation}\label{2error}
  E_t = \di_L g_t^2-g_t(a)^2\,\dd \mu_t+\di_R g_t^2-g_t(b)^2\,\dd \mu_t=o(\st)\di_I 
\lvert g_t'\rvert ^2\,\dd \mu_t+o(1)\di_L \lvert g_t'\rvert ^2\,\dd \mu_t.  
\end{equation}
Therefore,
\begin{equation}\label{maj3}
    \begin{split}
        \fii\Big(\di f_t^2 \mu_t\Big)&=\fii\Big(\di \fn^2\,\dd \mu_t\Big)+\fii'\Big(\di\fn^2\,\dd \mu_t\Big)\bigg(\di_R 2(g_t-g_t(b))\,\dd \mu_t+\di_L 2(g_t-g_t(a))\,\dd \mu_t +E_t\bigg)\\
        &+O(1)\bigg(\di_R 2(g_t-g_t(b))\,\dd \mu_t+\di_L 2(g_t-g_t(a))\,\dd \mu_t +E_t\bigg)^2.\\
\end{split}
\end{equation}
But
\begin{align*}
\begin{split}
\di \fn^2\,\dd \mu_t=1+\di_I g_t^2\,\dd \mu_t+\at(2g_t(b)+g_t(b)^2)+\bt(2g_t(a)+g_t(a)^2)=1+O(1)\sqrt{\di_I \lvert g_t' \rvert^2\,\dd \mu_t}.
\end{split}
\end{align*}
Hence \begin{equation}\label{derive}
    \fii'\bigg(\di \fn^2\,\dd \mu_t\bigg)=1+O(1)\sqrt{\di_I \lvert g_t' \rvert^2\,\dd \mu_t}.
\end{equation}
Furthermore,
\begin{equation}\label{produitdroite}
    \begin{split}
        \sqrt{\di_I \lvert g_t' \rvert^2\,\dd \mu_t}\di_R g_t-g_t(b)\,\dd \mu_t &\leq t^\frac{3}{4}\di_I \lvert g_t' \rvert^2\,\dd \mu_t+t^{-\frac{3}{4}}O(t)\di_R \lvert g_t'\rvert ^2\,\dd \mu_t\\&=o(\st)\di_I \lvert g_t' \rvert^2\,\dd \mu_t+o(1)\di_R \lvert g_t' \rvert^2\,\dd \mu_t.
    \end{split}
\end{equation}
where we used \eqref{order1}. In the same way, we obtain
\begin{equation}\label{produitgauche}
    \sqrt{\di_I \lvert g_t' \rvert^2\,\dd \mu_t}\di_L g_t-g_t(a)\,\dd \mu_t=o(\st)\di_I \lvert g_t' \rvert^2\,\dd \mu_t+o(1)\di_L \lvert g_t' \rvert^2\,\dd \mu_t.
\end{equation}
Finally, combining \eqref{2error}, \eqref{maj3}, \eqref{derive}, \eqref{produitdroite} and \eqref{produitgauche} yields
\begin{equation*}
\begin{split}
    \fii\Big(\di f_t^2 \mu_t\Big)&=\fii\Big(\di \fn^2\,\dd \mu_t\Big)+2\di_R g_t-g_t(b)\,\dd \mu_t+2\di_L g_t-g_t(a)\,\dd \mu_t\\
    &+o(\st)\di_I \lvert g_t' \rvert^2\,\dd \mu_t+o(1)\di_{I^c} \lvert g_t' \rvert^2\,\dd \mu_t.
    \end{split}
\end{equation*}
Combining this with \eqref{logmoyennedroite} and \eqref{logmoyennegauche}, and recalling that $g_t'=f_t'$ yields
\begin{equation*}
    \begin{split}
        \tiny\dfrac{\di \fii(f_t^2)\,\dd \mu_t-\fii\Big(\di f_t^2\,\dd \mu_t\Big)}{\di \lvert f_t' \rvert^2\,\dd \mu_t}&=\tiny\dfrac{\di \fii(\fn^2)\,\dd \mu_t-\fii\Big(\di \fn ^2\,\dd \mu_t\Big)+o(\st)\di_I \lvert g_t' \rvert^2\,\dd \mu_t+o(1)\di_{I^c} \lvert g_t' \rvert^2\,\dd \mu_t}{\di \lvert f_t' \rvert^2\,\dd \mu_t}\\
        &=\tiny\dfrac{\di \fii(\fn^2)\,\dd \mu_t-\fii\Big(\di \fn ^2\,\dd \mu_t\Big)+o(1)\di_{I^c} \lvert g_t' \rvert^2\,\dd \mu_t}{\di_I \lvert f_t' \rvert^2\,\dd \mu_t+\di_{I^c} \lvert f_t' \rvert^2\,\dd \mu_t}+o(\st)\\
        &\leq \tiny\dfrac{\di \fii(\fn^2)\,\dd \mu_t-\fii\Big(\di \fn ^2\,\dd \mu_t\Big)}{\di_I \lvert f_t' \rvert^2\,\dd \mu_t}+o(\st).
    \end{split}
\end{equation*}
Thus,
$$2\ct=\dfrac{\di \fii(\fn^2)\,\dd \mu_t-\fii\Big(\di \fn^2\,\dd \mu_t\Big)}{\di \lvert \fn' \rvert^2\,\dd \mu_t}+o(\st).$$
Consequently, we may assume without loss of generality that $f_t$ is constant on $R$ and on $L$, and, by Step $2$, we therefore have $f_t\xrightarrow[]{L^\infty(\mathbb{R})}1$.

\vspace{15pt}
\noindent\textbf{Step 4.} \textit{Linearizing, rescaling and upper bound.}

In this fourth step, we rescale $f_t$ once more by a factor of $1+o(1)$, so that $\di_If_t^2\,\dd \mu_t=\mu_t(I)$, i.e. $\dfrac{1}{b-a}\di_If_t(x)^2 \,\dd x=1$.
Note that we can still assume that $f$ is constant on $R$ and on $L$, that $\di \lvert f_t' \rvert^2\,\dd \mu_t \to 0$ and that $f_t\xrightarrow[]{L^\infty(\mathbb{R})}1$. The function $g_t$ is still defined as $f_t-1$, so that $$2\di_I g_t(x)\,\dd x+\di_Ig_t(x)^2\,\dd x=0.$$ 

\begin{align}\label{eqn1} \nonumber
    2\ct+o(\st)&=\dfrac{\di \fii(f_t^2)\,\dd \mu_t-\fii\Big(\di f_t^2\,\dd \mu_t\Big)}{\di \lvert f_t' \rvert^2\,\dd \mu_t}\\ \nonumber
    &=\dfrac{\di_I \fii(f_t^2)\,\dd \mu_t+\di_{I^c}\fii(f_t^2)\,\dd \mu_t-\fii\Big(\di f_t^2\,\dd \mu_t\Big)}{\di \lvert f_t' \rvert^2\,\dd \mu_t}\\
    &\leq 2C_{\mathrm{LS}}(\mu_0)+\dfrac{\di_{I^c}\fii(f_t^2)\,\dd \mu_t-\fii\Big(\di f_t^2\,\dd \mu_t\Big)}{\di \lvert f_t' \rvert^2\,\dd \mu_t}
    \end{align}
where we used  
\begin{align*}
    \begin{split}
        \frac{1}{b-a}\di_I \fii(f_t(x)^2)\,\dd x &\leq \fii\Bigg(\overbrace{\dfrac{1}{b-a}\di_I f_t(x)^2\,\dd x}^{=\,1}\Bigg)+\frac{2C_{\mathrm{LS}}(\mu_0)}{b-a}\di_I \lvert f_t'(x) \rvert^2\,\dd x\\
        &=\frac{2C_{\mathrm{LS}}(\mu_0)}{b-a}\di_I \lvert f_t'(x) \rvert^2\,\dd x.
    \end{split}
\end{align*}
A Taylor expansion of $\fii$ yields
\begin{align}\label{eqn2}
\begin{split}
\di_{R}\fii(f_t^2)\,\dd \mu_t=r_t\fii(1+2g_t(b)+g_t(b)^2)=r_t\big(2g_t(b)+3g_t(b)^2\big)+O(\st\lVert  g_t \rVert^3_\infty)\\
\di_{L}\fii(f_t^2)\,\dd \mu_t=\bt\fii(1+2g_t(a)+g_t(a)^2)=\bt\big(2g_t(a)+3g_t(a)^2\big)+O(\st\lVert  g_t \rVert^3_\infty)
\end{split}
\end{align}
Moreover, since
\begin{equation*}
    \begin{split}
        \di f_t^2\,\dd \mu_t&=1+2\di g_t\,\dd \mu_t +\di g_t^2\,\dd \mu_t\\
        &=1+2\di_{I^c}g_t\,\dd \mu_t+\di_{I^c}g_t^2\,\dd \mu_t\\
        &=1+\at(2g_t(b)+g_t(b)^2)+\bt(2g_t(a)+g_t(a)^2),
    \end{split}
\end{equation*}
it follows that
\begin{equation}\label{eqn3}
    \begin{split}
        \fii\Big(\di f_t^2\,\dd \mu_t\Big)=\at(2g_t(b)+g_t(b)^2)+\bt(2g_t(a)+g_t(a)^2)+O(t\lVert g_t \rVert^2_\infty).
    \end{split}
\end{equation}
Combining \eqref{eqn1}, \eqref{eqn2} and \eqref{eqn3} yields
\begin{equation*}
    2\ct\leq 2C_{\mathrm{LS}}(\mu_0)+\dfrac{2\at g_t(b)^2+2\bt g_t(a)^2 +O(\st\lVert  g_t \rVert^3_\infty)+O(t\lVert g_t \rVert^2_\infty)}{\di_I \lvert g_t'\rvert^2\,\dd \mu_t}+o(\st).
\end{equation*}
But $g_t=f_t-1$ vanishes somewhere on $I$ as $\dfrac{1}{b-a}\di_If_t^2\,\dd \mu_t=1$. Hence $\lVert g_t \rVert^2_\infty\leq O(1)\di_I \lvert g_t' \rvert^2$ and we have
$$\dfrac{O(\st\lVert  g_t \rVert^3_\infty)+O(t\lVert g_t \rVert^2_\infty)}{\di_I \lvert g_t' \rvert^2}=o(\st)$$
from which it follows that
\begin{equation}\label{preresult}
\begin{split}
\ct &\leq C_{\mathrm{LS}}(\mu_0)+\dfrac{\at g_t(b)^2+\bt g_t(a)^2}{\di_I \lvert g_t' \rvert^2\,\dd \mu_t}+o(\st)\\
&\leq C_{\mathrm{LS}}(\mu_0)+\dfrac{\beta g_t(b)^2+\alpha g_t(a)^2}{(1+o(1))\di_I \lvert g_t' \rvert^2\,\dd \mu_0}\st+o(\st)\\
&= C_{\mathrm{LS}}(\mu_0)+\dfrac{\beta g_t(b)^2+\alpha g_t(a)^2}{\di_I \lvert g_t' \rvert^2\,\dd \mu_0}\st+o(\st)
\end{split}
\end{equation}
where $\alpha=\lim \dfrac{\bt}{\st}$ and $\beta =\lim \dfrac{\at}{\st}$.

\medskip

\noindent It remains to determine the limit of $\dfrac{\alpha g_t(b)^2+\beta g_t(a)^2}{\di_I \lvert g_t' \rvert^2\,\dd \mu_t}$ to obtain an upper bound on $\ct$.
Set $\lambda_t=\sqrt{\di \lvert g_t'\rvert^2\,\dd \mu_0}$ and $h_t=\dfrac{g_t-\di g_t\,\dd \mu_t}{\sqrt{\di \lvert g_t' \rvert^2\,\dd \mu_0}}$.
Using a linearization argument once more,
\begin{align*}
    &\dfrac{\di \fii(f_t^2)\,\dd \mu_t-\fii\Big(\di f_t^2\,\dd \mu_t\Big)}{\di \lvert g_t' \rvert^2\,\dd \mu_t}=\dfrac{2\di g_t^2\,\dd \mu_t-2\Big( \di g_t\,\dd \mu_t\Big)^2+O(\lVert g_t \lVert^3_\infty)}{\di \lvert g_t' \rvert^2\,\dd \mu_t}\\
    &=\dfrac{2\mathrm{Var_{\mu_t}}(g_t)}{(1+o(1))\di \lvert g_t' \rvert^2\,\dd \mu_0}+o(1)=\dfrac{2\di h_t^2\,\dd \mu_t}{\di \lvert h_t' \rvert^2\,\dd \mu_0}+o(1)=\dfrac{2\mathrm{Var}_{\mu_0}(h_t)}{\di \lVert h_t' \rVert^2\,\dd \mu_0}+o(1)
\end{align*}
where we used $\di \lvert h_t'\rvert^2\,\dd \mu_0=\Theta(1)$ and $\di h_t\,\dd \mu_0=o(1)$ as $\di h_t d \mu_t=0$ and $\lVert h_t \rVert_\infty = O(1)$.
Hence $$C_{\mathrm{P}}(\mu_0) \leq C_{\mathrm{LS}}(\mu_0) \leq \liminf \ct \leq \liminf \dfrac{\mathrm{Var}_{\mu_0}(h_t)}{\di \lVert h_t' \rVert^2\,\dd \mu_0}.$$
From the stability of the Poincaré inequality (the argument is given in the proof of Theorem \ref{1d}), it follows that for any sequence $t_n \to 0$, $(h_{t_n}-\di h_{t_n}\,\dd \mu_0)_n$ is relatively compact in $H^1(I)$ and its set of accumation points is contained in $\{u,-u\}$, where $$u(x)=\dfrac{(b-a)\sqrt{2}}{\pi} \sin\Big(\frac{\pi}{b-a}\big(x-\frac{a+b}{2}\big)\Big) ~ \text{ for $x\in I$.}$$

\medskip

Take a sequence $\tn \to 0$ such that $h_\tn \xrightarrow[n\to \infty]{H^1(I)}u$. For $t \in T=\{\tn, n\in \mathbb{N}\}$, we have 
\begin{align*}
    g_t&=\di g_t\,\dd \mu_t+ \lambda_t h_t =\di g_t\,\dd \mu_t + \lambda_t\Big( \di h_t\,\dd \mu_0 + u+\varepsilon_t\Big)\\
    &=c_t +\lambda_t u + \lambda_t\varepsilon_t
\end{align*}
where $\varepsilon_t \xrightarrow[t\to 0,~
t\in T]{H^1(I)}0$ and $c_t=\di g_t\,\dd \mu_t+\lambda_t\di h_t\,\dd \mu_0$.
Since $\di 2g_t+g_t^2\,\dd \mu_0$ and $\di u\,\dd \mu_0=0$, it follows that 
\begin{align*}
    0&=2c_t+2\lambda_t\di \varepsilon_t\,\dd \mu_0+c_t^2+\lambda_t^2\di u^2\,\dd \mu_0+\lambda_t^2\di \varepsilon_t^2\,\dd \mu_0\\ &+2\lambda_t^2\di u\varepsilon_t\,\dd \mu_0+2\lambda_tc_t\di u\,\dd \mu_0+2\lambda_tc_t\di \varepsilon_t\,\dd \mu_0.
\end{align*}
Therefore, $c_t=o(\lambda_t)$. Consequently, as $u(b) \neq 0$,
\begin{align*}
    \dfrac{g_t(b)^2}{\di \lvert g_t' \rvert^2\,\dd \mu_0} \sim \dfrac{\lambda_t^2u(b)^2}{\di \lambda_t^2\lvert u'\rvert^2\,\dd \mu_0+\di\lambda_t^2 u'\varepsilon_t'\,\dd \mu_0}\xrightarrow[t\in T]{t \to 0}\dfrac{u(b)^2}{\di \lvert u' \rvert^2\,\dd \mu_0}.
\end{align*}
Similarly,
$$\dfrac{g_t(a)^2}{\di \lvert g_t' \rvert^2\,\dd \mu_0}\xrightarrow[t \in T]{t \to 0}\dfrac{u(a)^2}{\di \lvert u' \rvert^2\,\dd \mu_0}.$$
Recalling \eqref{preresult}, we deduce that 
\begin{align*}
    C_{\mathrm{LS}}(\mu_\tn) &\leq C_{\mathrm{LS}}(\mu_0)+\dfrac{\beta g_\tn(b)^2+\alpha g_\tn(a)^2}{\di \lvert g_\tn' \rvert^2\,\dd \mu_0}\sqrt{\tn}+o(\sqrt{\tn})\\
    &\leq C_{\mathrm{LS}}(\mu_0)+\dfrac{\beta u(b)^2+\alpha u(a)^2}{\di \lvert u' \rvert^2\,\dd \mu_0}\sqrt{\tn}+o(\sqrt{\tn}).
\end{align*}
By the same argument, the same bound holds whenever $(t_n)_{n \in \mathbb{N}}$ is sequence such that $t_n \to 0$ and $h_\tn-\di h_\tn\,\dd \mu_0 \xrightarrow[]{H^1(I)} -u$. As a result, the upper bound holds along all sequences $\tn \to 0$ and we have therefore shown that
$$\ct \leq C_{\mathrm{LS}}(\mu_0)+\dfrac{\beta u(b)^2+\alpha u(a)^2}{\di \lvert u' \rvert^2\,\dd \mu_0}\st+o(\st).$$

\noindent\textbf{Step 5.} \textit{Lower bound.}

In this final step, we show the converse inequality. The proof is the same as for the Poincaré inequality.
Extending $u$ to $\mathbb{R}$ by making it continuous and constant on $R$ and $L$ yields
\begin{align*}
    &\ct \geq C_{\mathrm{P}}(\mu_t) \geq \dfrac{\mathrm{Var}_{\mu_t}(u)}{\di \lvert u' \rvert ^2\,\dd \mu_t}=\dfrac{\di u^2\,\dd \mu_t -\big(\at u(b)+\bt u(a) \big)^2}{\di \lvert u' \rvert ^2\,\dd \mu_t}\\
    &= \dfrac{\di_I u^2\,\dd \mu_t+\at u(b)^2+ \bt u(a)^2}{\di_I \lvert u' \rvert ^2\,\dd \mu_t}+O(t)=\overbrace{C_{\mathrm{P}}(\mu_0)}^{= C_{\mathrm{LS}}(\mu_0)}+\dfrac{\at u(b)^2+ \bt u(a)^2}{\di_I \lvert u' \rvert ^2\,\dd \mu_t}+O(t)\\
    &=C_{\mathrm{LS}}(\mu_0)+\dfrac{\beta u(b)^2+\alpha u(a)^2}{\di \lvert u' \rvert^2\,\dd \mu_0}\st+o(\st).
\end{align*}
Thus,
$$\ct = C_{\mathrm{LS}}(\mu_0)+\sqrt{2}(b-a)\pi^{-\frac{3}{2}}\big(V''(a^-)^{-\frac{1}{2}}+{V''(b^+)^{-\frac{1}{2}}}\big)\st +o(\st)
$$
and the proof is complete.
\end{proofof}

\bibliographystyle{alpha}
\bibliography{biblio.bib}

\end{document}